\documentclass[12pt, reqno]{amsart}

\usepackage{amsmath}
\usepackage{amsthm}
\usepackage{amssymb}
\usepackage{amsrefs}
\usepackage{mathrsfs}
\usepackage[usenames]{color}
\usepackage{bm}
\usepackage{enumitem}
\newlist{altenumerate}{enumerate}{1}
\setlist*[altenumerate]{label=\textbf{(p\arabic*)}, resume=alt}

\newtheorem{thm}{}[section]
\newtheorem{theorem}[thm]{Theorem}
\newtheorem{corollary}[thm]{Corollary}
\newtheorem{lemma}[thm]{Lemma}
\newtheorem{proposition}[thm]{Proposition}

\theoremstyle{remark}
\newtheorem{definition}[thm]{Definition}

\newtheorem{remark}[thm]{Remark}

\newtheorem{problem}[thm]{Problem}

\numberwithin{equation}{section}
\allowdisplaybreaks

\newcommand{\DD}{\ensuremath{\mathbb{D}}}
\newcommand{\Fou}{\ensuremath{\mathcal{F}}}
\newcommand{\EE}{\ensuremath{\mathcal{E}}}
\newcommand{\AAA}{\ensuremath{\mathcal{A}}}
\newcommand{\PP}{\ensuremath{\mathcal{P}}}
\newcommand{\LL}{\ensuremath{\mathcal{L}}}

\newcommand{\YY}{\ensuremath{\mathbb{Y}}}
\newcommand{\GG}{\ensuremath{\mathcal{G}}}
\newcommand{\HH}{\ensuremath{\mathcal{H}}}
\newcommand{\JJ}{\ensuremath{\mathcal{J}}}
\newcommand{\BB}{\ensuremath{\mathcal{B}}}
\newcommand{\OO}{\ensuremath{\mathcal{O}}}
\newcommand{\SSB}{\ensuremath{\mathbb{S}}}
\newcommand{\ssb}{\ensuremath{\bm{s}}}

\newcommand{\uu}{\ensuremath{\bm{u}}}
\newcommand{\zz}{\ensuremath{\bm{z}}}
\newcommand{\ww}{\ensuremath{\bm{w}}}
\newcommand{\WW}{\ensuremath{\mathcal{W}}}
\newcommand{\NN}{\ensuremath{\mathbb{N}}}
\newcommand\rr{\ensuremath{\bm{r}}}
\newcommand\ee{\ensuremath{\bm{e}}}
\newcommand\xx{\ensuremath{\bm{x}}}
\newcommand\yy{\ensuremath{\bm{y}}}
\newcommand{\XX}{\ensuremath{\mathbb{X}}}
\newcommand{\QQ}{\ensuremath{\mathbb{Q}}}

\newcommand{\FF}{\ensuremath{\mathbb{F}}}
\newcommand{\dom} {\mathop{\mathrm{D}}}
\newcommand{\rang} {\mathop{\mathrm{R}}}
\newcommand{\supp} {\mathop{\mathrm{supp}}}

\begin{document}

\title[A dichotomy for subsymmetric bases]{A dichotomy for subsymmetric basic sequences with applications to Garling spaces}

\author[F. Albiac]{F. Albiac}\address{Department of Mathematics, Statistics and Computer Sciences, and InaMat2\\ Universidad P\'ublica de Navarra\\
Pamplona 31006\\ Spain}
\email{fernando.albiac@unavarra.es}

\author[J. L. Ansorena]{J. L. Ansorena}\address{Department of Mathematics and Computer Sciences\\
Universidad de La Rioja\\
Logro\~no 26004\\ Spain}
\email{joseluis.ansorena@unirioja.es}

\author[S. Dilworth]{S. J. Dilworth}\address{Department of Mathematics\\
University of South Carolina\\
Columbia SC 29208 \\ USA}
\email{dilworth@math.sc.edu}

\author[Denka Kutzarova]{Denka Kutzarova}
\address{
Department of Mathematics University of Illinois at Urbana-Champaign\\
Urbana, IL 61801 \\ USA \\
and Institute of Mathematics and Informatics\\ Bulgarian Academy of Sciences\\
Sofia\\ Bulgaria.
}
\email{denka@math.uiuc.edu}

\subjclass[2010]{46B15, 46B20, 46B45}

\keywords{Garling spaces, Lorentz spaces, sequence spaces, symmetric basic sequence, subsymmetric basic sequence}

\begin{abstract}

Our aim in this article is to contribute to the study of the structure of subsymmetric basic sequences in Banach spaces (even, more generally, in quasi-Banach spaces). For that we introduce the notion of \emph{positionings} and develop new tools which lead to a dichotomy theorem that holds for general spaces with subsymmetric bases. As an illustration of how to use this dichotomy theorem we obtain the classification of all subsymmetric sequences in certain types of spaces.
To be more specific, we show that Garling sequence spaces have a unique symmetric basic sequence but no symmetric basis and that these spaces have a continuum of subsymmetric basic sequences.
\end{abstract}

\thanks{F. Albiac acknowledges the support of the Spanish Ministry for Economy and Competitivity under Grant MTM2016-76808-P for \emph{Operators, lattices, and structure of Banach spaces}. F. Albiac and J.~L. Ansorena acknowledge the support of the Spanish Ministry for Science, Innovation, and Universities under Grant PGC2018-095366-B-I00 for \emph{An\'alisis Vectorial, Multilineal y Aproximaci\'on}. S.~J. Dilworth acknowledges the support from the National Science Foundation under Grant Number DMS--1361461. Denka Kutzarova acknowledges the support from the Simons Foundation Collaborative under Grant No 636954. F. Albiac, S. J. Dilworth and Denka Kutzarova would like to thank the Isaac Newton Institute, Cambridge, for support and hospitality during the programme \emph{Approximation, Sampling and Compression in Data Science}, where work on this paper was undertaken. This work was supported by ESPRC Grant no EP/K032208/1.}

\maketitle

\section{Introduction and background}\label{intro}
\noindent Let $\XX$ be a (real or complex) separable Banach space. One of the main problems in the isomorphic theory of Banach spaces is the classification of the basic sequences of a certain type in $\XX$. This question is formulated in a proper way using the notion of equivalence of basic sequences. Recall that a sequence $(\xx_j)_{j=1}^{\infty}$ in a Banach space is a basic sequence if it is a (Schauder) basis of its closed linear span; two basic sequences $(\xx_j)_{j=1}^{\infty}$ and $(\yy_j)_{j=1}^{\infty}$ in $\XX$ are said to be \emph{equivalent} provided a series $\sum_{j=1}^{\infty}a_{j}\, \xx_{j}$ converges in $\XX$ if and only if $\sum_{j=1}^{\infty}a_{j}\yy_{j}$ does.

The most important category of sequences in which this classification is studied is that of symmetric sequences. A basic sequence $(\xx_j)_{j=1}^\infty$ is \emph{symmetric} if the rearranged sequence $(\xx_{\pi(j)})_{j=1}^{\infty}$ is equivalent to $(\xx_j)_{j=1}^\infty$ for any permutation $\pi$ of $\NN$. This class of sequences includes the canonical unit vector basis of the $\ell_{p}$ spaces and $c_{0}$. Closely related to symmetry is the notion of subsymmetry. A basic sequence $(\xx_j)_{j=1}^\infty$ in $\XX$ is said to be \emph{subsymmetric} if it is unconditional, i.e., the rearranged sequence $(\xx_{\pi(j)})_{j=1}^{\infty}$ is also a basic sequence for any permutation $\pi$ of $\NN$, and $(\xx_{\phi(j)})_{j=1}^{\infty}$ is equivalent to $(\xx_j)_{j=1}^\infty$ for any increasing map $\phi\colon\NN\to\NN$.

The question whether a symmetric basic sequence exists in every Banach space was a driving force in the development of the theory for many decades. And, despite the fact that the question was solved in the negative by Tsirelson in 1974 (\cite{Cirel1974}) it motivated a plethora of new interesting problems.

The class of subsymmetric basic sequences is more general (see, e.g., \cites{Singer1962, KadetsPel1962}).
In practice, the only feature that one needs about symmetric basic sequences in many situations is their subsymmetry, to the extent that when symmetric bases were introduced these two concepts were believed to be equivalent until Garling \cite{Garling1968} provided a counterexample that disproved it.

However, subsymmetric bases, far from being just a capricious generalization of symmetric bases, played a relevant role by themselves within the general theory. Indeed, the study of Banach spaces with a non-symmetric subsymmetric basis led to the solution of major problems in the field. For example, the first arbitrarily distortable space constructed by Schlumprecht \cite{Schlumprecht1991} has one such basis. Other important landmarks whose attainment was inspired by techniques related to this concept are the solution of the unconditional basic sequence problem by Gowers and Maurey \cite{GM1993} and the distortion of $\ell_p$ spaces by Odell and Schlumprecht \cite{OS1994}. The construction of all these spaces is based on techniques that differ greatly from the aforementioned example of Garling. Besides, since subsymmetric bases are just unconditional spreading sequences, they appear naturally in those contexts where spreading models are applied.

For background, let us next briefly outline a few well-known facts and  milestone results in the classification of symmetric and subsymmetric basic sequences in Banach spaces.

Neither the space nowadays known as the original Tsirelson space nor its dual have any subsymmetric basic sequences (see \cites{FJ1974,Cirel1974}). Per contra, the unit vector system is
the unique subsymmetric basic sequence of $\ell_p$ ($1\le p<\infty$) and $c_0$
(see, e.g., \cite{AAW2018}*{Proposition 2.14} and \cite{AK2016}*{Proposition 2.1.3}).

Kadec and Pe{\l}czy{\'n}ski proved in \cite{KadetsPel1962} that $L_p[0,1]$ does not have a subsymmetric basis for $1\le p<\infty$, $p\not=2$. The set of indices $q$ for which $L_p[0,1]$ has a basic sequence equivalent to the unit vector system of $\ell_q$ is the interval $[p,2]$ if $1\le p < 2 $ and the finite set $\{2,p\}$ if $2\le p<\infty$ (see, e.g., \cite{AK2016}*{Theorem 6.4.18}).

The classification of symmetric basic sequences in Lorentz sequence spaces $d(\ww,p)$, for $1\le p<\infty$, being $\ww=(w_n)_{n=1}^\infty\in c_0\setminus\ell_1$ a positive non-increasing weight, was obtained in \cite{ACL1973}. Here, Altshuler et al.\ proved that the space $d(\ww,p)$ has a unique symmetric basis and that the classification of its symmetric basic sequences depends on $\ww$. In the case when $\ww$ is submultiplicative, i.e.,
\[
\sup_{m,n}\frac{\displaystyle \sum_{n=1}^{m n} w_j} {\displaystyle \left(\sum_{j=1}^{m} w_j \right) \left(\sum_{j=1}^{n} w_j\right)}<\infty,
\]
the space $d(\ww,p)$ has exactly two symmetric (and subsymmetric) basic sequences, namely the unit vector bases of $d(\ww,p)$ and $\ell_p$; otherwise, $d(\ww,p)$ has more than two (non-equivalent) symmetric basic sequences, and there are examples of weights $\ww$ for which $d(\ww,p)$ has infinitely many symmetric basic sequences.

Let $h_F$ denote the separable part of the Orlicz sequence space $\ell_F$. It is known \cite{Lindberg1973} that every subsymmetric basic sequence in $h_F$ is equivalent to the unit vector system of an Orlicz space $h_G$. In particular, every subsymmetric basic sequence is symmetric. Lindenstrauss and Tzafriri showed in \cite{LT1971} that if $\lim_{t\to 0} tF^{\prime}(t)/F(t)$ exists then $h_F$ has a unique symmetric basis. In the same paper, an Orlicz sequence space with exactly two symmetric basic sequences is provided. The same authors gave in \cite{LT1973} a sufficient condition for $h_F$ to have uncountably many subsymmetric basic sequences. More recently, the article \cite{DilworthSari2008} contains an intricate construction of an Orlicz sequence space with a countably infinite collection of symmetric basic sequences. Sari's work \cite{Sari2007}, in turn, discusses the structure with respect to a domination order relation of the set of symmetric basic sequences in an Orlicz sequence space.


The examples above show that all possible situations that one can have a priori, actually do occur, namely:
\begin{enumerate}[label={(\alph*)}]
\item there are no symmetric basic sequences;

\item there is (up to equivalence) a unique symmetric basic sequence;

\item there exist finitely many non-equivalent symmetric basic sequences; and

\item there are infinitely many nonequivalent symmetric basic sequences.
\end{enumerate}

The case when a Banach space has a unique symmetric basic sequence deserves special attention. Note that amongst all the 
aforementioned examples, only $\ell_p$ and $c_0$ have a unique symmetric basic sequence. Altshuler constructed in \cite{A1977} a Banach space with a symmetric basis which does not contain any subspace isomorphic to $\ell_p$ or $c_0$, and whose symmetric basic sequences are all equivalent. Another remarkable example was provided by Read \cite{R1981}. He answered a question of Lindenstrauss and Tzafriri by constructing a space with exactly two symmetric bases (up to equivalence). More precisely, Read proved that every symmetric basic sequence is equivalent either to the unit vector basis of $\ell_1$ or to one of the two symmetric bases of his space. It was remarked in \cite{KMP2012} that a careful look at the papers of Altshuler and Read shows that their proofs work similarly for the more general case of all subsymmetric basic sequences. Let us restate some of the questions raised in \cite{KMP2012} about the richness of subsymmetric sequences.

\begin{problem} Does there exist a Banach space in which all subsymmetric basic sequences are equivalent to one basis, and that basis is not symmetric?
\end{problem}

\begin{problem}
Does there exist a Banach space with, up to equivalence, precisely two subsymmetric bases, at least one of them is not symmetric?
\end{problem}

In addition, we may also ask the following.
\begin{problem}
Does there exist a Banach space with, up to equivalence, precisely $n$ subsymmetric basic sequences such that at least one of them is not symmetric? Here $n$ is a natural number greater than $1$.
\end{problem}

With an eye on this kind of questions, in Section~\ref{SubSymSection} we develop new tools which lead to a dichotomy theorem for general Banach spaces with a subsymmetric basis. We prove that every subsymmetric sequence in such a space is either of a certain canonical form or it dominates a normalized sequence that vanishes in the supremum norm. In the symmetric setting, this type of dichotomy has its roots at least as far back as 1973 in the work of Altshuler, Casazza, and Lin \cite{ACL1973}, where it is effectively used to study the structure of symmetric basic sequences in Lorentz sequence spaces. In the spreading model setting, this dichotomy was also present in the work of Beauzamy and Laprest\'e \cite{BL1984}. However, the subsymmetric case has not been understood that well. In order to bridge this gap in the theory we introduce the necessary novel notion of \emph{positionings} which yields our main result, namely the dichotomy theorem.

In Section~\ref{GarlingSection} the dichotomy theorem is used to investigating the structure of subsymmetric basic sequences in Garling-like sequence spaces. Specifically, we show that there exists an entire new class of spaces with a unique symmetric basic sequence. The Banach spaces we investigate here were introduced in \cite{AAW2018} by extending the pattern in the example of Garling cited above. Because of that they were called \emph{Garling sequence spaces}. Although Garling sequence spaces $g(\ww,p)$ for $1\le p<\infty$ and $\ww\in c_0\setminus\ell_1$, share with the spaces $\ell_p$ the property that each of their symmetric basic sequences is equivalent to the unit vector system of $\ell_p$, the reader should be warned that $g(\ww,p)$ is far from behaving like $\ell_p$. For instance, it is known (see \cites{AAW2018,AALW2018}) that Garling sequence spaces have a unique subsymmetric basis which is not symmetric, hence they possess no symmetric basis. The authors showed in \cite{AAW2018} that Garling sequence spaces behave in some sense like Lorentz sequence spaces, after which Garling sequence spaces are modeled replacing symmetry with subsymmetry. Showing that Garling sequence spaces have infinitely many nonequivalent subsymmetric basic sequences will evince that Garling sequence spaces and Lorentz sequence spaces have structures that are further apart than one may think.
The uniqueness of symmetric basic sequence will be proved in Section~\ref{GarlingSymSection}, while the construction of uncountably many subsymmetric basic sequences will be achieved in Section~\ref{GarlingSubSymSection}.

Although we are mainly interested in Banach spaces, for the sake of completeness we will make sure that our discussion remains valid in the more general setting of (not necessarily locally convex) quasi-Banach spaces. Throughout this note we use standard Banach space theory notation and terminology, as can be found, e.g., in \cite{AK2016}. For the convenience of the reader, though, we next single out the terminology that will be most heavily used. As is customary, we write $c_{00}$ for the space of all scalar sequences with finitely many nonzero entries. We will write $\FF$ for the real or complex field. Given a countable set $J$, $(\ee_j)_{j\in J}$ denotes the unit vector system on $J$, i.e., $\ee_j=(\delta_{k,j})_{k\in J}$ for every $j\in J$ where $\delta_{k,j}=1$ if $n=k$ and $\delta_{k,j}=0$ otherwise. The unit vector system on $J$ regarded inside a quasi-Banach space $\XX\subseteq \FF^J$, will be denoted by $\EE[\XX]$. The \emph{support} of $x=(a_j)_{j=1}^\infty\in\FF^J$ is the set $\supp(x)=\{ j\in J\colon a_j\not=0\}$. The \emph{coordinate projection} of $x\in\FF^J$ on $A\subseteq J$ is defined by $S_A(x)=(a_j \chi_A(j))_{j\in A}$. If $J=\NN$ and $n\in\NN$ $S_n\colon\FF^\NN\to\FF^\NN$ denotes the coordinate projection on the $n$ first coordinates. If $J$ is totally ordered and $x$, $y\in\FF^J$ are such that $i<j$ for every $i\in\supp(x)$ and $j\in\supp(y)$ we write $x\prec y$.

A basis of a quasi-Banach space $\XX$ is a sequence $\mathcal B= (\xx_j)_{j=1}^\infty$ in $\XX$ such that for every $x\in\XX$ there is unique sequence $(a_n)_{n=1}^\infty$ in $\FF$ such that $x=\sum_{j=1}^\infty a_j\, \xx_j$. If $(\xx_j)_{j=1}^\infty$ is a basis, the map $\xx_k^*\colon \XX\to \FF$ given by $\sum_{j=1}^\infty a_j\, \xx_j\mapsto a_k$ is called its \emph{$k$th coordinate functional}, and
\[
\Fou \colon\XX\to\FF^\NN, \quad x \mapsto (\xx_j^*(x))_{j=1}^\infty
\]
its \emph{coefficient transform}. The \emph{support} of $x$ with respect to the basis $\BB$ is the set $\supp(x):=
\supp(\Fou(x))$ of indices corresponding to its nonzero entries. Given $x$ and $y\in\XX$ we write $x\prec y$ if $\Fou(x)\prec \Fou(y)$. The closed linear span of a subset $Z$ of a quasi-Banach space $\XX$ will be denoted by $[Z]$. A basic sequence $(\yy_k)_{k=1}^\infty$ is \emph{semi-normalized} if
$
0<c=\inf_k \Vert \yy_k\Vert$ and $C =\sup_k \Vert \yy_k\Vert<\infty.
$
If $c=C=1$ the basic sequence is said to be normalized. Given two basic sequences $\BB_1=(\xx_k)_{k=1}^\infty$ and $\BB_2=(\yy_k)_{k=1}^\infty$ in quasi-Banach spaces $\XX$ and $\YY$, respectively, we say that $\BB_1$ \emph{dominates} $\BB_2$, and write $\BB_2 \lesssim \BB_1$, if there is a bounded linear operator $T\colon [\xx_k\colon k\in\NN] \to [\yy_k] \colon k\in\NN]$ such that $T(\xx_n)=\yy_n$ for all $n\in\NN$. Quantitatively, if $\Vert T\Vert \le C$ for some positive $C$ we say that $\BB_1$ $C$-\emph{dominates} $\BB_2$ and write $\BB_2\lesssim_C \BB_1$. Note that $\BB_1$ and $\BB_2$ are equivalent if and only if $\BB_2\lesssim \BB_1$ and $\BB_1\lesssim \BB_2$. If it is the case, we write $\BB_1\simeq \BB_2$.

If the basis $(\xx_j)_{j=1}^\infty$ is unconditional, then $\sum_{j=1}^\infty a_j\, \xx_j$ converges unconditionally for every $x=\sum_{j=1}^\infty a_j\, \xx_j\in\XX$. So, unconditional bases can be indexed with an infinite countable set instead of $\NN$. A \emph{quasi-Banach lattice} on a countable set $J$ is a quasi-Banach space $\XX\subseteq \FF^J$ such that $\Vert x\Vert \le \Vert y \Vert$ whenever $|x|\le |y|$. By a \emph{sign} we mean a scalar of modulus one. It is known that a linearly independent family $\BB=(\xx_j)_{j\in J}$ satisfying $[\BB]=\XX$ is unconditional if and only if for every sequence of signs $\varepsilon=(\epsilon_j)_{j\in J}$ there is a bounded linear operator $M_\varepsilon\colon \XX\to\XX$ such that $M_\epsilon(\xx_j)=\varepsilon_j\, \xx_j$ for every $j\in J$. In the case when $\Vert M_\varepsilon\Vert\le C$ the basis is said to be $C$-unconditional. Every unconditional basis is $C$-unconditional for some $C$, and the optimal constant $C=1$ is attained under renorming. Thus, an unconditional basis $(\xx_j)_{j\in J}$ induces a lattice structure on $\XX$: given $x$, $y\in\XX$, we say that $|x|\le |y|$ if $|\xx_j^*(x)|\le |\xx_j^*(y)|$ for every $j\in J$. Given $0<p<\infty$, a quasi-Banach lattice is said to be \emph{$p$-convex} if there is a constant $C$ such that
\[
\left\Vert \left(\sum_{k=1}^n |x_k|^p \right)^{1/p}\right\Vert \le C \left(\sum_{k=1}^n \Vert x_k\Vert^p \right)^{1/p}
\]
for every $n\in\NN$ and $(x_k)_{k=1}^n$ in $\XX$. The optimal constant $C$ will be called the $p$-convexity constant of the lattice. Notice that any Banach lattice is $1$-convex with constant $1$. Given $0<q<\infty$ the \emph{$q$-convexification} of the quasi-Banach lattice $\XX$ on $J$ is the quasi-Banach lattice consisting of all $x\in \FF^J$ such that $|x|^q\in \XX$, endowed with the quasi-norm $x\mapsto \Vert |x|^q\Vert^{1/q}$.

Given infinite-dimensional quasi-Banach spaces $\XX$ and $\YY$, $\LL(\XX,\YY)$ denotes the space of bounded linear operators from $\XX$ into $\YY$, and put $\LL(\XX)=\LL(\XX,\YY)$. The symbol $\XX\simeq \YY$ means that $\XX$ and $\YY$ are isomorphic.

We denote by $\Pi$ the set of all permutations of $\NN$ and by $\OO$ the set of increasing functions from $\NN$ into $\NN$. $\Pi_n$ will be the set of all permutations of $\NN[n]$, where
$\NN[n]$ is the set of the first $n$ natural numbers, i.e.,
\[
\NN[n]=\{1,\dots, k,\dots,n\}.
\]

Given a function $\phi$ we denote by $\dom(\phi)$ its domain and by $\rang(\phi)$ its range. Note that a function in $\OO$ is univocally determined by its range.

\section{A dichotomy for subsymmetric basic sequences}\label{SubSymSection}
\noindent
We start this section by introducing some definitions concerning a quasi-Banach space $\SSB$ with a subsymmetric basis $(\ssb_j)_{j=1}^\infty$.

\begin{definition}A quasi-norm $\Vert \cdot \Vert $ on $\SSB$ is A said to be subsymmetric with respect to $(\ssb_j)_{j=1}^\infty$ if
\begin{equation}\label{eq:subsym}
\Vert x\Vert= \left\Vert \sum_{j=1}^\infty \varepsilon_j \, \ssb_j^*(x) \, \ssb_{\phi(j)} \right\Vert,
\quad x\in\SSB, \, |\varepsilon_n |= 1, \, \phi\in\OO.
\end{equation}
In this case we will say that the basis $(\ssb_j)_{j=1}^\infty$ is $1$-subsymmetric. \end{definition}

A quasi-norm of $\SSB$ is subsymmetric if and only if the basis $(\ssb_j)_{j=1}^\infty$ is $1$-unconditional and for every increasing map $\phi\colon A\subseteq \NN\to \NN$ the linear operator $V_\phi\in\LL(\SSB)$ defined by
\begin{equation}\label{eq:Vphi}
x=\sum_{j=1}^\infty a_j \, \ssb_j\mapsto
V_\phi(x)=\sum_{j\in\dom(\phi)} a_j \, \ssb_{\phi(j)}
\end{equation}
satisfies $\Vert V_\phi\Vert\le 1$ (see \cite{Ansorena2018}). Notice that if $\phi\in\OO$ then $V_\phi$ is an isometric embedding.

Recall that a quasi-norm is said to be a $p$-norm, $0<p\le 1$, if it is $p$-subadditive, i.e.,
\[
\Vert x+y\Vert^p \le \Vert x\Vert^p +\Vert y \Vert^p, \quad x, y\in\SSB.
\]
Combining the Aoki-Rolewicz theorem (see, e.g., \cite{KPR}) with the techniques developed in \cite{Ansorena2018} yields that any quasi-Banach space with a subsymmetric basis can always be endowed with an equivalent $p$-norm which is subsymmetric with respect to the basis, for some $0<p\le 1$.

A \emph{block basic sequence} with respect to the basis $(\ssb_j)_{j=1}^\infty$ is a sequence $(\yy_k)_{k=1}^\infty$ of non-zero vectors such that $\yy_k\prec \yy_{k+1}$ for all $k\in\NN$. Block basic sequences are a particular type of basic sequences which play a key role in the theory of Banach spaces. As we next show, if we focus on subsymmetric basic sequences its role is even more significant.

\begin{lemma}\label{lem:boundedsupport} Let $\SSB$ be a quasi-Banach space with a subsymmetric basis $(\ssb_j)_{j=1}^\infty$. Suppose $(\yy_n)_{n=1}^\infty$ is a semi-normalized block basic sequence with respect to $(\ssb_j)_{j=1}^\infty$ such that
\[
N:=\sup_n |\supp(\yy_n)|<\infty.
\]
Then $\BB$ is equivalent to $(\ssb_j)_{j=1}^\infty$.
\end{lemma}
\begin{proof} Without loss of generality we assume that $\SSB$ is equipped with a $1$-subsymmetric $p$-norm. For every $n\in\NN$ we have
\[
c:=\inf_m \Vert \yy_m\Vert
\le \left(\sum_{j\in \supp(\yy_n)} |\ssb_j^*(\yy_n)|^p\right)^{1/p}
\le N^{1/p} \sup_{j\in \supp(\yy_n)} |\ssb_j^*(\yy_n)|.
\]
Hence, there is $\phi\colon\NN\to\NN$ such that $|\ssb_{\phi(n)}^*(\yy_n)|\ge D:=c^{-1}N^{1/p}$ for every $n\in\NN$. By subsymmetry,
\[
(\ssb_j)_{j=1}^\infty\lesssim_1(\ssb_{\phi(n)})_{n=1}^\infty \lesssim_D (\yy_n)_{n=1}^\infty.
\]
For each $k=1$, \dots, $N$, let $A_k=\{ n \in\NN \colon |\supp(\yy_n)|\ge k\}$. There is $(\phi_k)_{k=1}^N$ such that $\dom(\phi_k)=A_k$ and $\supp(\yy_n)=\{\phi_k(n) \colon n \in A_k\}$.
Put $C=\sup_m \Vert \yy_m\Vert$. Then for any $(a_n)_{n=1}^\infty\in c_{00}$,
\[
\left\Vert \sum_{n=1}^\infty a_n \, \yy_n\right\Vert^p
\le D^p\sum_{k=1}^N
\left\Vert \sum_{n\in A_k}^\infty a_n \, \ssb_{\phi_k(n)}\right\Vert^p
\le N D^p \left\Vert \sum_{n=1}^\infty a_n \, \ssb_n\right\Vert^p.
\]
Hence, if $E:=N^{1/p} C$ we obtain $(\yy_n)_{n=1}^\infty\lesssim_E (\ssb_j)_{j=1}^\infty$.
\end{proof}

\begin{proposition}\label{prop:1}Let $(\yy_n)_{n=1}^\infty$ be a subsymmetric basic sequence in a quasi-Banach space with a basis $(\xx_j)_{j=1}^\infty$. Then $(\yy_n)_{n=1}^\infty$ is equivalent to a block basic sequence of $(\xx_j)_{j=1}^\infty$.
\end{proposition}
\begin{proof}Since $(\yy_n)_{n=1}^\infty$ is semi-normalized, $(\xx_j^*(\yy_n))_{n=1}^\infty$ is bounded for every $j\in\NN$. Using Cantor's diagonal argument there is $\phi\in\OO$ such that
$(\xx_j^*(\yy_{\phi(n)}))_{n=1}^\infty$ converges for every $j\in\NN$. Set
\[
\zz_n=\yy_{\phi(2n)}- \yy_{\phi(2n-1)}, \quad n\in\NN.
\]
Since $\lim_j \xx_j^*(\zz_n)=0$, by the gliding hump technique, $(\zz_n)_{n=1}^\infty$ has a subsequence equivalent to a block basic sequence of $(\xx_j)_{j=1}^\infty$. By Lemma~\ref{lem:boundedsupport}, $(\zz_n)_{n=1}^\infty$ is equivalent to $(\yy_n)_{n=1}^\infty$, then it is subsymmetric and we are done.
\end{proof}

\subsection{Basic sequences generated by a seed} Among block basic sequences, basic sequences generated by a vector are of particular interest when studying the geometry of Banach spaces with a symmetric basis (see \cite{ACL1973}). If the basis is merely subsymmetric, we need a slightly more sophisticated concept. Roughly speaking, basic sequences generated by a vector are constructed by means of a recursive process consisting of adding coefficients to a vector at each step. The idea is that if the basis is not symmetric the position where the new coefficients are placed plays a significant role. As we intend to add larger coefficients first, our process connects with methods from greedy approximation with respect to bases, whose language we borrow. 

Given a Banach space $\SSB$ with a basis $(\ssb_j)_{j=1}^\infty$, we define the \emph{greedy ordering} of a vector $x\in \SSB$ as the unique one-to-one map $\rho\colon\NN\to\NN$ satisfying
\begin{itemize}
\item $(|\ssb_{\rho(j)}^*(x)|)_{j=1}^\infty$ is non-increasing;
\item if $j\le k$ and $|\ssb_{\rho(j)}^*(x)|=|\ssb_{\rho(k)}^*(x)|$, then $\rho(j)\le \rho(k)$;
\item if $\supp(x)$ is finite then $\rho(\NN)=\NN$; and
\item if $\supp(x)$ is infinite, then $\rho(\NN)=\supp(x)$.
\end{itemize}
If $\rho$ is the greedy ordering of $x$ and $n\in\NN$, we define the \emph{$n$th greedy set} of $x$ by
\[
A_n(x)=\{\rho(j) \colon 1\le j \le n\}
\]
and the \emph{$n$th greedy sum} of $x=\sum_{j=1}^\infty a_j \, \ssb_j$ by
\[
\GG_n(x)=\sum_{j=1}^n a_{\rho(j)} \, \ssb_{\rho(j)}=\sum_{j\in A_n(x)} a_j\, \ssb_j.
\]

We say that $y\in\SSB$ is a \emph{right-shift} of $x\in\SSB$ if there is $\phi\in\OO$ such that $V_\phi(x)=y$, where $V_\phi$ is defined as in \eqref{eq:Vphi}. If the basis $(\ssb_j)_{j=1}^\infty$ is symmetric, then a block basic sequence $(\xx_n)_{n=1}^\infty$ such that
$\xx_n$ is a right-shift of $\GG_n(x)$ for all $n\in\NN$ is equivalent to a block basic sequence $(\yy_n)_{n=1}^\infty$ such that $\yy_n$ is is a right-shift of
\[
\HH_n(x):=\sum_{j=1}^n a_{\rho(j)} \, \ssb_{j}
\]
for all $n\in\NN$. If the basis $(\ssb_j)_{j=1}^\infty$ is merely subsymmetric, the sequence $(\HH_n(x))_{n=1}^\infty$ no longer does the job. Nonetheless, for each $n\in\NN$ there is a permutation $\pi_n$ of $\NN[n]$ such that $\GG_n(x)$ is a right-shift of
\[
\JJ_n(x):=\sum_{j=1}^n a_{\rho(j)} \, \ssb_{\pi_n(j)}
\]
for all $n\in\NN$. Then, the block basic sequence $(\zz_n)_{n=1}^\infty$ is equivalent to $(\xx_n)_{n=1}^\infty$ provided that $\zz_n$ is a right-shift of $\JJ_n(x)$ for all $n\in\NN$. This idea naturally leads to introducing the notions of a positioning and a seed.

Throughout this section and the forthcoming Section \ref{sect:22} we will deal with a quasi-Banach space $\SSB$ equipped with a subsymmetric basis $(\ssb_j)_{j=1}^\infty$, and we will assume, without loss of generality, that the quasi-norm in $\SSB$ is $p$-subadditive and subsymmetric. In this case, $V_\phi$ is an isometric embededing for all $\phi\in\OO$ and, hence, $\Vert y\Vert=\Vert x\Vert$ whenever $y$ is a right-shift of $x\in\SSB$.

\begin{definition}
A \emph{positioning} will be a sequence $(d_n)_{n=1}^\infty$ such that $d_n\in\NN[n]$ for all $n\in\NN$. Given a positioning $\eta=(d_n)_{n=1}^\infty$ we recursively define a sequence $\pi[\eta]=(\pi_n)_{n=1}^\infty$ with $\pi_n\in\Pi_n$ for all $n\in\NN$. Starting with $\pi_0=\emptyset$ we put
\[
\pi_n(j)=\begin{cases} \pi_{n-1}(j) & \text{ if } \pi_{n-1}(j)<d_n, \\ 1+\pi_{n-1}(j) & \text{ if } \pi_{n-1}(j)\ge d_n, \\ d_n & \text{ if } j=n.\end{cases}
\]
A \emph{seed} will be a pair $(f,\eta)$, where $f$ is a sequence of scalars and $\eta$ is a positioning. Given a seed $(f,\eta)$ we recursively define the families $\uu[f,\eta]=(u_n)_{n=1}^\infty$
and $\uu_l[f,\eta]=(u_{m,n})_{1\le m \le n}$ in $\SSB$ by
\[
u_{m,n}=\sum_{j=1}^m a_j \, \ssb_{\pi_n(j)}, \quad u_n=u_{n,n}=\sum_{j=1}^n a_j \, \ssb_{\pi_n(j)},
\]
where $\pi[\eta]=(\pi_n)_{n=1}^\infty$ and $f=(a_j)_{j=1}^\infty$.
\end{definition}

Note that if $\eta=(d_n)_{n=1}^\infty$ is a positioning and $\pi[\eta]=(\pi_n)_{n=1}^\infty$, then
\begin{equation}\label{eq:etafrompi}
d_n=|\{j\in \NN[n] \colon \pi_n(j)\le \pi_n(n)\}|, \quad n\in\NN.
\end{equation}

In the case when $\eta=(d_n)_{n=1}^\infty$ is the standard positioning given by $d_n=n$ for $n\in\NN$ we put $\uu[f]=\uu[f,\eta]$ and $\uu_l[f]=\uu_l[f,\eta]$. Note that in this particular case, $(u_n)_{n=1}^\infty$ is the sequence of partial sums of the formal series $\sum_{j=1}^\infty a_j\, \ssb_j$ and $u_{m,n}=u_m$ for every $m\le n$.

The following result summarizes some early properties of the concepts we have defined.

\begin{lemma}\label{lemma:first}Let $\eta$ be a positioning. Given $f$, $\tilde f\in\FF^\NN$,
put
\[
\uu_l[f,\eta]=(u_{m,n})_{m\le n},\; \uu[f,\eta]=(u_n)_{n=1}^\infty,\; \text{and}\;\, \uu[\tilde f, \eta]=(\tilde u_n)_{n=1}^\infty.\]
Then:
\begin{enumerate}[label={(\alph*)}]
\item\label{lemma:first:a} $\supp(u_{m,n})\subseteq\NN[n]$ for all $m\le n$.

\item\label{lemma:first:b} If $m\in\NN$ then $\uu[f-S_m(f),\eta]=(u_{n}-u_{m,n})_{n=1}^\infty$ (with the convention that $u_{m,n}=u_n$ for $m>n$).

\item\label{lemma:first:c} If $|\tilde f|\le |f|$ then $\vert \tilde u_n\vert \le \vert u_n\vert$ for all $n\in\NN$. Consequently, $\Vert \tilde u_n\Vert \le \Vert u_n\Vert$ for all $n\in\NN$.

\item\label{lemma:first:d} If $m \le n $, then $u_{m,n}$ is a right-shift of $u_{m}$.

\item\label{lemma:first:e} $(\Vert u_n\Vert)_{n=1}^\infty$ is a non-decreasing sequence.

\item\label{lemma:first:f} $(\Vert u_{n}-u_{m,n} \Vert)_{0\le m\le n}$ increases with $n$ and decreases with $m$ (with the convention that $u_{0,n}=0$).
\end{enumerate}
\end{lemma}
\begin{proof} \ref{lemma:first:a}, \ref{lemma:first:b}, and \ref{lemma:first:c} are clear.

\ref{lemma:first:d} For each $j\in\NN$, let $\sigma_j\in\OO$ be such that $R(\sigma_j)=\NN\setminus\{ j\}$. Then,
if $m\le n-1$ and $\eta=(d_n)_{n=1}^\infty$,
\[
V_{\sigma_{d_n}}(u_{m,n-1})=u_{m,n}.
\]
From here, we proceed by induction.

\ref{lemma:first:e} Let $m\le n$. Using \ref{lemma:first:d} and $1$-uncondionality we get
\[\Vert u_m\Vert=\Vert u_{m,n}\Vert\le \Vert u_n\Vert.\]

\ref{lemma:first:f} follows from \ref{lemma:first:b}, \ref{lemma:first:c} and \ref{lemma:first:e}.
\end{proof}
Let $(f, \eta)$ be a seed and, as usual, put $\uu[f, \eta]=(u_n)_{n=1}^\infty$ and
$\uu_l[f, \eta]=(u_{m,n})_{n=1}^\infty$. With Lemma~\ref{lemma:first}~\ref{lemma:first:e} and Lemma~\ref{lemma:first}~\ref{lemma:first:f} in mind we define
\[
E_0[f,\eta]=\sup_n \Vert u_n\Vert=\lim_n \Vert u_n\Vert,
\]
and, for $m\in\NN$,
\[
E_m[f,\eta]=\sup_{n\ge m} \Vert u_{n}-u_{m,n}\Vert=\lim_n \Vert u_{n}-u_{m,n}\Vert.
\]
Note that $(E_m[f,\eta])_{m=0}^\infty$ is non-increasing. We have
\[
E_\infty[f,\eta]=\inf_m E_m[f,\eta]=\lim_m E_m[f,\eta].
\]

\begin{proposition} Let $(f,\eta)$ be a seed. Then $E_0[f,\eta]<\infty$ if and only if $E_\infty[f,\eta]<\infty$.
\end{proposition}

\begin{proof} Since $E_\infty[f,\eta]\le E_0[f,\eta]$, we need only prove the ``if'' part. Assume that $E_\infty[f,\eta]<\infty$. Then $E_m[f,\eta]<\infty$ for some $m$. By Lemma~\ref{lemma:first}~\ref{lemma:first:d}, if $\uu_l[f,\eta]=(u_{m,n})_{m\le n}$ and $\uu[f,\eta]=(u_n)_{n=1}^\infty$,
\[
\Vert u_n\Vert\le \Vert u_{n}-u_{m,n} \Vert+\Vert u_{m,n} \Vert=\Vert u_{n}-u_{m,n} \Vert+\Vert u_m\Vert.
\]
Therefore $E_0[f,\eta]\le E_m[f,\eta]+\Vert u_m\Vert<\infty$.
\end{proof}

\begin{lemma} If $E_\infty[f,\eta]=0$, then $f\in c_0$. \end{lemma}

\begin{proof} If $f=(a_j)_{j=1}^\infty$, then $|a_j|\le \Vert u_j\Vert$ for all $j\in\NN$. Applying this inequality to $f-S_{j-1}(f)$ and taking into account Lemma~\ref{lemma:first}~\ref{lemma:first:c},
\[
|a_j|\le \Vert u_j -u_{j-1,j} \Vert \le E_{j-1}[f,\eta],\quad j\ge 2.
\]
Letting $j$ tend to infinity finishes the proof.
\end{proof}

We say that a block basic sequence $(\xx_k)_{k=1}^\infty$ is \emph{generated by a seed} $(f,\eta)$ if there is an unbounded sequence $\nu=(n_k)_{k=1}^\infty$ in $\NN$ such that, with the usual notation $\uu[f,\eta]=(u_n)_{n=1}^\infty$, $\xx_k$ is a right-shift of $u_{n_k}$. By subsymmetry, all block basic sequences generated by the same seed who also share the unbounded sequence $\nu$ used in the above definition are isometrically equivalent. Let us construct, among them, a particular block basic sequence. Given $n\in\NN\cup\{0\}$ we define $\tau_n\in\OO$ by
\[
\tau_n(j)=j+n.
\]
Put $q_k=\sum_{i=1}^{k-1} n_i$ for all $k\in\NN$ (with the convention that $\sum_{i=1}^0 x_i=0$). Let
\[
\BB[f,\eta,\nu]=(V_{\tau_{q_k}}(u_{n_k}))_{k=1}^\infty.
\]
Note that every block sequence generated by a seed is in fact generated by a seed $((a_j)_{j=1}^\infty,\eta)$ such that $a_1\not=0$. We will refer to such seeds as \emph{proper}. Imposing the seed to be proper allows us to ensure that $u_n\not=0$ for every $n\in\NN$ so that $\BB[f,\eta,\nu]$ is a block basic sequence.

\begin{lemma}\label{lemma:second} Suppose that $|f|\le|g|$. Then $\BB[f,\eta,\nu]\lesssim_1 \BB[g,\eta,\nu]$.
\end{lemma}
\begin{proof} It s straightforward from Lemma~\ref{lemma:first}~\ref{lemma:first:c}.
\end{proof}

Given $0<p<\infty$ we set $\overline p=\min\{1,p\}$.
\begin{lemma}\label{SubSymFromSeed}
Let $(f,\eta)$ be a proper seed and $\mu=(m_k)_{k=1}^\infty$ be an unbounded sequence of natural numbers. Suppose that
\[
\sum_{k=1}^\infty (E_{m_k}[f,\eta])^{\overline p}<\infty.
\]
\begin{enumerate}[label={(\alph*)}]
\item\label{SubSymFromSeed:a} If $\nu=(n_k)_{k=1}^\infty$ verifies $m_k\le n_k$ for all $k\in\NN$ then $\BB[f,\eta,\mu]\simeq\BB[f,\eta,\nu]$.
\item\label{SubSymFromSeed:b} $\BB[f,\eta,\mu]$ is a subsymmetric basic sequence.
\end{enumerate}
\end{lemma}
\begin{proof}
\ref{SubSymFromSeed:a} Put
\[
\uu[f,\eta]=(u_n)_{n=1}^\infty,\; \uu_l[f,\eta]=(u_{m,n})_{n=1}^\infty,\; \BB[f,\eta,\mu]=(\yy_k)_{k=1}^\infty,
\]
and, for $k\in\NN$, $q_k=\sum_{i=1}^{k-1} n_i$. By Lemma~\ref{lemma:first}~\ref{lemma:first:d}, $\xx_k=V_{\tau_{q_k}}(u_{m_k,n_k})$ is a right-shift of $u_{m_k}$ for every $k\in\NN$. By Lemma~\ref{lemma:first}~\ref{lemma:first:a}, $\BB:=(\xx_k)_{k=1}^\infty$ is disjointly supported. Hence, $\BB$ is a basic sequence isometrically equivalent to $\BB[f,\eta,\mu]$. We have
\[
\Vert \yy_k-\xx_k\Vert=\Vert V_{q_k}(u_{n_k}-u_{m_k,n_k})\Vert =
\Vert u_{n_k}-u_{m_k,n_k}\Vert \le E_{m_k}[f,\eta].
\]
The principle of small perturbations (see, e.g., \cite{AK2016}*{Theorem 1.3.9}) yields that $\BB[f,\eta,\nu]\simeq\BB\simeq\BB[f,\eta,\mu]$.

\ref{SubSymFromSeed:b} If $\BB$ be a subbasis of $\BB[f,\eta,\mu]$ then $\BB$ is isometrically equivalent to $\BB[f,\eta,\tilde\mu]$, where $\tilde\mu=(\tilde m_k)_{k=1}^\infty$ is a subsequence of $\mu$. Let $\nu=(\max\{m_k,\tilde m_k\})_{k=1}^\infty$. Since
\[
\sum_{k=1}^\infty (E_{\tilde m_k}[f,\eta])^{\overline p}\le \sum_{k=1}^\infty (E_{m_k}[f,\eta])^{\overline p}<\infty,
\]
applying \ref{SubSymFromSeed:a} yields
$\BB[f,\eta,\mu]\simeq \BB[f,\eta,\nu] \simeq \BB[f,\eta,\tilde\mu]\simeq\BB$.
\end{proof}

\begin{proposition}\label{SubSymFromamenableSeed}Let $(f,\eta)$ be a seed with $E_\infty[f,\eta]=0$. Then
\begin{enumerate}[label={(\alph*)}]
\item\label{SubSymFromamenableSeed:a} There is an increasing sequence $\nu$ of natural numbers such that $\BB[f,\eta,\nu]$ is subsymmetric.

\item\label{SubSymFromamenableSeed:b} Assume that $\mu$ and $\nu$ are increasing sequences of natural numbers such that both $\BB[f,\eta,\mu]$ and $\BB[f,\eta,\nu]$ are subsymmetric. Then $\BB[f,\eta,\mu]\simeq \BB[f,\eta,\nu]$.
\end{enumerate}

\end{proposition}
\begin{proof} Our hypothesis yields the existence of an unbounded sequence $\tilde\mu=(m_k)_{k=1}^\infty$ such that
$
\sum_{k=1}^\infty (E_{m_k}[f,\eta])^{\overline p}<\infty.
$
Moreover, $\tilde\mu$ could be chosen to be a subsequence of a given unbounded sequence. By Lemma~\ref{SubSymFromSeed}~\ref{SubSymFromSeed:b}, $\BB[f,\eta,\tilde\mu]$ is subsymmetric and so \ref{SubSymFromamenableSeed:a} holds.

In order to prove \ref{SubSymFromamenableSeed:b} we pick $\tilde\mu$ as above, which is a subsequence of $\mu$. Then, we pick a subsequence $\tilde\nu=(n_k)_{k=1}^\infty$ of $\nu$ with $m_k\le n_k$ for all $k$. Invoking Lemma~\ref{SubSymFromSeed}~\ref{SubSymFromSeed:a} we obtain
\[
\BB[f,\eta,\nu]\simeq \BB[f,\eta,\tilde\nu]\simeq \BB[f,\eta,\tilde\mu]\simeq \BB[f,\eta,\mu]. \qedhere
\]
\end{proof}

Proposition~\ref{SubSymFromamenableSeed} allows us to assign to any seed $(f,\eta)$ with $E_\infty[f,\eta]=0$ an equivalence class of subsymmetric basic sequences in $\SSB$. If $\BB$ is an element of that equivalence class we say that $\BB$ is a \emph{subsymmetric basic sequence generated by} $(f,\eta)$. We make a stop en route for proving that the subsymmetric basis $(\ssb_j)_{j=1}^\infty$ of $\SSB$ can be recovered using this procedure.
\begin{proposition}\label{Prop:EventuallyNull}
Let $f$ be an eventually null sequence whose first term is not null. Then $E_\infty[f,\eta]=0$ for any positioning $\eta$. Moreover any basic sequence generated by $(f,\eta)$ is equivalent to $(\ssb_j)_{j=1}^\infty$.
\end{proposition}
\begin{proof} We have $E_m[f,\eta]=0$ for $m$ large enough. Now the result follows from Lemma~\ref{lem:boundedsupport}.
\end{proof}

\subsection{Basic sequences generated by a vector}\label{sect:22}Let us consider the trivial positioning $\eta_0=(n)_{n=1}^\infty$. Given a sequence $f\in\FF^\NN$ we put
\[
\uu_l[f]=\uu_l[f,\eta_0], \uu[f]=\uu[f,\eta_0], E_m[f]=E_m[f,\eta_0]\; \text{for}\; m\in\NN\cup\{0,\infty\}.
\]
A basic sequence generated by the seed $(f,\eta_0)$ is said to be generated by $f$. A basic sequence will be said to be generated by a vector $x\in\SSB$ if it is generated by the coefficient transform $\Fou(x)$ of $x$.

\begin{proposition}\label{prop:2}Let $f$ be a sequence in $\FF^\NN$. Then $E_\infty[f,\eta_0]=0$ if and only if there is $x\in\SSB$ such that $\Fou(x)=f$. Moreover, if the subsymmetric basis $(\ssb_j)_{j=1}^\infty$ of $\SSB$ is boundedly complete and $E_0(f)<\infty$ then $E_\infty[f,\eta]=0$.
\end{proposition}
\begin{proof}Put $f=(a_j)_{j=1}^\infty$, $\uu[f]=(u_n)_{n=1}^\infty$ and $\uu_l[f]=(u_{m,n})_{m\le n}$. We have $u_n=\sum_{j=1}^n a_j\, \ssb_j$ for every $n\in\NN$ and $u_{m,n}=u_m$ for every $m$, $n\in\NN$ with $m\le n$. Hence,
\[
E_m[f]=\sup_{n\ge m+1} \left\Vert \sum_{j=m+1}^n a_j \, \ssb_j\right\Vert, \quad m\in\NN\cup\{0\}.
\]
Therefore $E_\infty[f,\eta]=0$ if and only if $\sum_{j=1}^\infty a_j \, \ssb_j$ is a Cauchy series. In turn, $\sum_{j=1}^\infty a_j \, \ssb_j$ is a Cauchy series if and only if there $x\in \SSB$ such that $x=\sum_{j=1}^\infty a_j \, \ssb_j$.

In the case when $(\ssb_j)_{j=1}^\infty$ is boundedly complete and $E_0(f)<\infty$, there is $x\in\SSB$ such that $x=\sum_{j=1}^\infty a_j \, \ssb_j$.
\end{proof}

\subsection{The dichotomy theorem}
We are ready to see the main result of this paper. In order to properly enunciate it, it will be convenient to state some additional notation. We say that a basic sequence $(\xx_k)_{k=1}^\infty$ is
\emph{uniformly null} if
\[
\lim_k \sup_j |\ssb_j^*(\xx_k)|=0.
\]
A seed $(f,\eta)$ is said to be non-negative and non-increasing if $f$ is.

\begin{theorem}\label{theorem:trichotomy:subsym}Let $\BB$ be a subsymmetric basic sequence in $\SSB$. Then:
\begin{enumerate}[label={(\alph*)}]
\item\label{theorem:trichotomy:subsym:a} Either $\BB$ is equivalent to a subsymmetric basic sequence generated by a non-negative and non-increasing seed $(f,\eta)$ with $E_\infty[f,\eta]=0$, or

\item\label{theorem:trichotomy:subsym:b} $\BB$ is equivalent to a block basic sequence and dominates a semi-normalized uniformly null block basic sequence.
\end{enumerate}
\end{theorem}

We will use the conventions that $A_0(x)=\emptyset$ and $\GG_0(x)=0$.

The \emph{fundamental function} $(\Phi(n))_{n=1}^\infty$ of a basis $(\ssb_j)_{j=1}^\infty$ is the sequence defined by
\[
\Phi(n)=\sup_{|A|\le n}\left\Vert \sum_{j\in A} \ssb_j\right\Vert.
\]
Note that if $(\ssb_j)_{j=1}^\infty$ is $1$-subsymmetric then $\Phi(|A|)=\Vert \sum_{j\in A} \ssb_j\Vert$ for every finite set $A\subseteq\NN$. The fundamental function of a basis is non-decreasing and, unless $(\ssb_j)_{j=1}^\infty$ is equivalent to the canonical basis of $c_0$, we have
\[
\lim_n \Phi(n)=\infty.
\]

If $A\subseteq\NN$, $j\in A$ and $|\{k\in A\colon k\le j\}|=d$, we say that $j$ is the $d$th element of $A$.

\begin{proof}[Proof of Theorem~\ref{theorem:trichotomy:subsym}]
By Propostion~\ref{prop:1} we can assume that $\BB=(\xx_k)_{k=1}^\infty$ is a block basic sequence with respect to the $1$-subsymmetric basis $(\ssb_j)_{j=1}^\infty$ of $\SSB$. Then, by unconditionality, we can assume that
\[
\ssb_j^*(\xx_n)\ge 0
\]
for every $j$, $n\in\NN$.

If $(\ssb_j)_{j=1}^\infty$ were equivalent to the canonical basis of $c_0$ so
would $\BB$ be. Thus we can assume that $\lim_n \Phi(n)=\infty$.

If there were $\delta>0$ such that
\[
\liminf_k \left\Vert \xx_k -\GG_n(\xx_k)\right\Vert> \delta,\quad n\in \mathbb N,
\]
then, replacing $\BB$ with a suitable subbasis, there would be an increasing sequence $(n_k)_{k=1}^\infty$ of non-negative integers such
that, if we put $\yy_k= \xx_k-\GG_{n_k}(\xx_k)$ for all $k\in\NN$,
\[
\inf_k \Vert \yy_k \Vert \ge\delta.
\]
Since the basis $(\ssb_j)_{j=1}^\infty$ is $1$-unconditional, the block basic sequence $\BB_1=(\yy_k)_{k=1}^\infty$ would be $1$-dominated by $\BB$. In particular, $\BB_1$ would be semi-normalized. Appealing to unconditionality again, for every $k\in\NN$ we would have
\begin{align*}
\sup_j| \ssb_j^*(\yy_k)|
&\le \min \{ |\ssb_j^*(\xx_k)| \colon j\in A_{n_k}(\xx_k) \} \\
&= \min \{ |\ssb_j^*(\xx_k)| \colon j\in A_{n_k}(\xx_k) \} \frac{ \left\Vert \sum_{j\in A_{n_k}(\xx_k)}\ssb_j\right\Vert}{\Phi(n_k)}\\
& \le \frac{\Vert \xx_k \Vert}{\Phi(n_k)}.
\end{align*}
Therefore, $\BB_1$ would be uniformly null. That is, we would be in the case \ref{theorem:trichotomy:subsym:b}. Hence, from now on we will assume that
\newlist{pes}{enumerate}{1}
\begin{altenumerate}
\item\label{dicho:p1} for every $\delta>0$ there are natural numbers $i=i(\delta)$ and $n=n(\delta)$ such that
$
\Vert \xx_k-\GG_{n}(\xx_k)\Vert\le \delta
$
for every $k\ge i$.
\end{altenumerate}
Notice that property \ref{dicho:p1} is preserved when passing to a subsequence.

For any $k\in\NN$, let $\rho_k$ be the greedy ordering of $\xx_k$. Set
\[
a_{k,n}=\ssb_{\rho_k(n)}^*(\xx_k), \quad k,n\in\NN.
\]
Note that
\begin{enumerate}[label=\textbf{(q\arabic*)}]
\item\label{dicho:q1} $(a_{k,n})_{n=1}^\infty$ is non-negative and non-increasing and
\item\label{dicho:q2} $\sup_{k} a_{k,1}<\infty$.
\end{enumerate}
If $\liminf_k a_{k,1}=0$, then, passing to a subsequence if necessary, $(\xx_k)_{k=1}^\infty$ would be uniformly null. Hence, we asume that
\begin{enumerate}[label=\textbf{(q\arabic*)}, resume]
\item\label{dicho:q3} $\inf_k a_{k,1}>0$.
\end{enumerate}
For each $k$ and $n\in\NN$, let $d_{k,n}$ be the position of $\rho_k(n)$ in $A_n(\xx_k)$. Next, we recursively construct
\begin{itemize}
\item a sequence $f=(a_n)_{n=1}^\infty\in\FF^\NN$,
\item a positioning $\eta=(d_n)_{n=1}^\infty$, and
\item a sequence $(\phi_n)_{n=1}^\infty$ in $\OO$ such that $\phi_m$ is a subsequence of $\phi_n$ whenever $m\le n$.
\end{itemize}
Let $n\in\NN$ and assume that $a_m$, $d_m$ and $\phi_m$ have been constructed for $m\le n-1$ (nothing is constructed if $n=1$). Taking into account \ref{dicho:q2}, by the Bolzano-Weierstrass theorem there is $\phi_n\in\OO$ (which can be chosen to be a subsequence of $\phi_{n-1}$), $a_n\in\FF$ and $d_n\in\NN[n]$ such that
\begin{enumerate}[label=\textbf{(q\arabic*)}, resume]
\item\label{dicho:q4} $d_{\phi_n(k),n}=d_n$ for all $k\in\NN$, and
\item\label{dicho:q5} $\lim_k a_{\phi_n(k),n}=a_n$.
\end{enumerate}

Combining \ref{dicho:q1}, \ref{dicho:q3} and \ref{dicho:q5} we obtain that
\begin{altenumerate}[resume]
\item\label{dicho:p2} $a_1>0$ and $(a_n)_{n=1}^\infty$ a non-increasing sequence of non-negative scalars.
\end{altenumerate}
Let $\pi[\eta]=(\pi_n)_{n=1}^\infty$. We infer from \ref{dicho:q4} that
\begin{enumerate}[label=\textbf{(q\arabic*)}, resume]
\item\label{dicho:q6} $\GG_n(\xx_k)-\GG_m(\xx_k)$ is a right-shift of $\sum_{j=m+1}^n a_{k,j} \, \ssb_{\pi_n(j)}$ whenever $0\le m\le n$ and $k\in\rang(\phi_n)$.
\end{enumerate}

Next we appeal to the classical Cantor's diagonal argument, i.e., we consider $\phi\in\OO$ defined by $\phi(n)=\phi_n(n)$ for all $n$. Then, replacing the block basic sequence $(\xx_k)_{k=1}^\infty$ with the equivalent block basic sequence $(\xx_{\phi(k)})_{n=1}^\infty$, we have
\begin{altenumerate}[resume]
\item\label{dicho:p3} $\lim_k a_{k,n}=a_n$ for all $n\in\NN$.
\item\label{dicho:p4} $\GG_n(\xx_k)-\GG_m(\xx_k)$ is a right-shift of $\sum_{j=m+1}^n a_{k,j} \, \ssb_{\pi_n(j)}$ whenever $0\le m\le n \le k$.
\end{altenumerate}

With the properties~\ref{dicho:p1}, \ref{dicho:p2}, \ref{dicho:p3} and \ref{dicho:p4} in hand, we are now in a position to complete the proof. Let us first see that $E_\infty[f,\eta]=0$. Let $\delta>0$ and pick $i$ and $m$ as in \ref{dicho:p1}. If $n\ge m$ and $k\ge \max\{i,n\}$, using \ref{dicho:p4} and unconditionality gives
\[
\left\Vert \sum_{j=m+1}^n a_{k,j} \, \ssb_{\pi_n(j)} \right\Vert
=\Vert \GG_{n}(\xx_k)-\GG_{m}(\xx_k)\Vert\le \Vert \xx_k-\GG_{m}(\xx_k)\Vert\le \delta.
\]
Let $\uu[f,\eta]=(u_n)_{n=1}^\infty$ and $\uu_l[f,\eta]=(u_{m,n})_{m\le n}$. Letting $k$ tend to infinity and appealing to \ref{dicho:p3} we obtain
\[
\Vert u_n-u_{m,n}\Vert = \left\Vert \sum_{j=m+1}^n a_{j} \, \ssb_{\pi_n(j)} \right\Vert\le\delta.
\]
Letting now $n$ tend to infinity we get $E_{m}[f,\eta]\le\delta$. Combining with \ref{dicho:p2}, we deduce that $(f,\eta)$ is a non-negative and non-decreasing seed with $E_\infty[f,\eta]=0$.

Finally, let us prove that $\BB$ is equivalent to a basic sequence generated by the seed $(f,\eta)$. Let $n(\cdot)$ and $i(\cdot)$ be defined by \ref{dicho:p1}. Given $k\in\NN$, set $n_k=n(2^{-k})$. Then, use property \ref{dicho:p2} to recursively construct an increasing sequence $(i_k)_{k=1}^\infty$ of natural numbers such that
\[
i_k\ge \max\{ i(2^{-k}), n_k\}
\]
and
\[
\Vert v_k-u_{n_k} \Vert \le 2^{-k},\] where

\[
v_k=\sum_{j=1}^{n_k} a_{i_k,j} \, \ssb_{\pi_{n_k}(j)}.
\]
By \ref{dicho:p4}, $\GG_{n_k}(\xx_{i_k})$ is a right-shift of $v_k$ for every $k\in\NN$. Using the principle of small perturbations and subsymmetry, the block basic sequences $\BB[f,\eta,\nu]$ and $(\GG_{n_k}(\xx_{i_k}))_{k=1}^\infty$ are equivalent. By construction, $\Vert \xx_{i_k}-\GG_{n_k}(\xx_{i_k}) \Vert \le 2^{-k}$ for every $k\in\NN$. Applying the principle of small perturbations once again we obtain that $(\GG_{n_k}(\xx_{i_k}))_{k=1}^\infty\simeq (\xx_{i_k})_{k=1}^\infty$. Combining and using the subsymmetry of $\BB$ yields $\BB[f,\eta,\nu]\simeq \BB$.
\end{proof}

\subsection{The case when the basis is symmetric}\label{SymSection}
Recall that if $(\ssb_j)_{j=1}^\infty$ is a symmetric basis of $\SSB$ then $\SSB$ can be equipped with a symmetric norm, i.e., a norm $\Vert\cdot\Vert$ such that
\[
\Vert x\Vert= \left\Vert \sum_{j=1}^\infty \varepsilon_j \, \ssb_j^*(x) \, \ssb_{\pi(j)} \right\Vert,\quad
\quad x\in\SSB, \, |\varepsilon_n |= 1, \, \pi\in\Pi,
\]
(see \cite{Singer1961}). So, whenever $(\ssb_j)_{j=1}^\infty$ is a symmetric basis, we will assume that the norm in $\SSB$ is symmetric.

As we next show, when dealing with symmetric bases, the technique based on seeds is unnecessarily complicated and, as Altshuler et al.\ \cite{ACL1973} did, it suffices to consider block basic sequences generated by a vector. We say that a vector $x\in\SSB$ is non-increasing (resp.\ non-negative) if its coefficient sequence $\Fou(x)$ is.

\begin{lemma}\label{lem:4}Let $\SSB$ be a quasi-Banach space with a symmetric basis. Then every block basic sequence generated by a seed $(f,\eta)$ with $E_\infty[f,\eta]=0$ is equivalent to a block basic sequence generated by a non-increasing and non-negative vector.
\end{lemma}

\begin{proof}Let $(f,\eta)$ be a seed with $E_\infty[f,\eta]=0$ and $\nu$ be an unbounded sequence of natural numbers. Let $f^*$ be the non-increasing rearrangement of $f$. By symmetry, $\BB[f,\eta,\nu]$ is isometrically $\BB[f^*,\nu]$, and $E_\infty[f^*,\eta]=0$. Applying Proposition~\ref{prop:2} puts an end to the proof.
\end{proof}

\begin{lemma}\label{lemma:sym:second} Suppose that all permutations of a sequence $(\xx_k)_{k=1}^\infty$ in a quasi-Banach space $\XX$ are subsymmetric bases. Then $(\xx_k)_{k=1}^\infty$ is a symmetric basis.
\end{lemma}
\begin{proof}Let $\pi\in\Pi$. We recursively construct $\phi\in\OO$ such that $\pi\circ\phi\in\OO$. Since both $(\xx_{\pi(k)})_{k=1}^\infty$ and $(\xx_k)_{k=1}^\infty$ are subsymmetric bases, we have $(\xx_{\pi(k)})_{k=1}^\infty\simeq (\xx_{\pi(\phi(k)})_{k=1}^\infty\simeq (\xx_k)_{k=1}^\infty$.
\end{proof}

Our following result is an improvement of the main result of \cite{FG2015}. Note that we do not impose the symmetric basis to be boundedly complete and that our result remains valid for non-locally convex spaces.
\begin{theorem}[cf. \cite{FG2015}*{Theorem 1.2}]\label{result:dichotomy:sym}
Suppose that $(\ssb_j)_{j=1}^\infty$ is a symmetric basis of the quasi-Banach space $\SSB$. Let $\BB$ be a subsymmetric basic sequence in $\SSB$. Then:

\begin{enumerate}[label={(\alph*)}]

\item\label{result:dichotomy:sym:a} Either $\BB$ is equivalent to a subsymmetric basic sequence generated by a vector, in which case $\BB$ is symmetric, or

\item\label{result:dichotomy:sym:b} $\BB$ is equivalent to a block basic sequence and dominates a semi-normalized uniformly null block basic sequence.
\end{enumerate}
\end{theorem}

\begin{proof}By Theorem~\ref{theorem:trichotomy:subsym} and Lemma~\ref{lem:4}, it suffices to prove that every subsymmetric basis $\BB$ generated by a seed $(f,\eta)$ with $E_\infty[f,\eta]=0$ is symmetric. Let
$
\mu=(m_k)_{k=1}^\infty$ be such that
$
\sum_{k=1}^\infty (E_{m_k}[f,\eta])^{\overline p}<\infty.
$
By Proposition~\ref{SubSymFromamenableSeed} and Lemma~\ref{SubSymFromSeed}~\ref{SubSymFromSeed:b}, it suffices to prove that $(\yy_k)_{k=1}^\infty=\BB[f,\eta,\nu]$ is symmetric. Let $\pi\in\Pi$. From the symmetry of $(\ssb_j)_{j=1}^\infty$ it follows that $\BB[f,\eta,\nu\circ\pi]\simeq(\yy_{\pi(k)})_{k=1}^\infty$. Applying again Lemma~\ref{SubSymFromSeed}~\ref{SubSymFromSeed:b} yields that, first $\BB[f,\eta,\nu\circ\pi]$ and then $(\yy_{\pi(k)})_{k=1}^\infty$, are subsymmetric basic sequences. We finish the proof by appealing to Lemma~\ref{lemma:sym:second}.
\end{proof}

\section{Symmetric and subsymmetric basic sequences in Garling sequence spaces}\label{GarlingSection}

\noindent Let $0< p<\infty$ and let $\ww=(w_j)_{j=1}^\infty$ be a non-increasing sequence of positive scalars. Given a sequence of (real or complex) scalars $x=(a_j)_{j=1}^\infty$ we put
\begin{equation}\label{GarlingNorm}
\Vert x \Vert_g =\Vert x \Vert_{g(\ww,p)} = \sup_{\phi\in\OO } \left( \sum_{j=1}^\infty |a_{\phi(j)}|^p w_j \right)^{1/p}.
\end{equation}
The \emph{Garling sequence space} $g(\ww,p)$ is the quasi-Banach space consisting of all sequences $x$ with $\Vert x \Vert_{g}<\infty$.

Notice that if we replace ``$\phi\in\OO$'' with ``$\phi\in\Pi$'' in \eqref{GarlingNorm} we obtain the norm defining the \emph{weighted Lorentz sequence space}
\[
d(\ww,p):=\left\{ (a_j)_{j=1}^\infty \in c_0 \colon \left( \sum_{j=1}^\infty (a_j^*)^p w_j \right)^{1/p}<\infty\right\},
\]
where $(b_j^*)_{i=1}^\infty$ denotes the decreasing rearrangement of $(b_j)_{j=1}^\infty$. Thus Garling sequence spaces $g(\ww,p)$ can be regarded as a variation of the weighted Lorentz sequence spaces $d(\ww,p)$.

We shall impose the further conditions $\ww\in c_0$ and $\ww\notin\ell_1$ to avoid the trivial cases $g(\ww,p)=\ell_p$ and $g(\ww,p)= \ell_\infty$, respectively. We will assume as well that $\ww$ is normalized, i.e., $w_1=1$. Thus, we put
\[
\WW:=\left\{(w_j)_{j=1}^\infty\in c_0\setminus\ell_1:1=w_1\geq w_2\ge\cdots \ge w_j \ge w_{j+1} \ge\cdots>0\right\}
\]
and we restrict our attention to \emph{weights} $\ww\in\WW$. Note that for every $(w_j)_{j=1}^\infty\in\WW$ and $N\in\NN$ we have
\begin{equation}\label{eq:weight}
\sum_{i=1}^\infty (w_i-w_{i+N})
= \sum_{i=1}^\infty \sum_{j=1}^{N} (w_{i+j-1} -w_{i+j})
=\sum_{j=1}^{N} w_j.
\end{equation}
The geometry of Garling sequence spaces has been studied in the locally convex range of $p$, i.e., for $p\ge 1$, in \cites{AAW2018,AADK2019}. Some of the results proved there can be transferred to non-locally convex Garling sequence spaces using that $g(\ww,p)$ is the $p$-convexification of the Banach lattice $g(\ww,1)$. Let us single out some properties of interest for the purposes of the present paper and leave the straightforward details for the reader.

\begin{theorem}\label{thm:garling1}Let $0<p<\infty$ and $\ww\in\WW$.
\begin{enumerate}[label=(\alph*)]
\item\label{thm:garling1:a} $\EE[g(\ww,p)]$ is a $1$-subsymmetric boundedly complete basis of the whole space $g(\ww,p)$.
\item\label{thm:garling1:b} $g(\ww,p)$ is a $p$-convex quasi-Banach lattice with constant one. In particular, we have $\BB\lesssim_1 \EE[\ell_p]$ for
every normalized basic sequence $\mathcal B$ of $g(\ww,p)$.
\item\label{thm:garling1:c} $g(\ww,p)$ is not $q$-concave por any $q<\infty$, i.e., $\ell_\infty$ is finitely representable in $g(\ww,p)$.
\item\label{thm:garling1:d} Every normalized uniformly null block basic sequence with respect to the unit vector system of $g(\ww,p)$ has, for any $C>1$, a subsequence $\BB$ such that $ \EE[\ell_p]\lesssim_C \BB$. Moreover, if $p\ge 1$ we can ensure that $\BB$ spans a complemented subspace of $g(\ww,p)$.
\end{enumerate}
\end{theorem}

\subsection{Garling spaces on $\QQ$} With the aim to classify the subsymmetric basic sequences of $g(\ww,p)$, we shall next introduce a variation of Garling sequence spaces.
Let $\ww\in\WW$ and $0<p<\infty$. The quasi-Banach lattice $g(\QQ,\ww,p)$ consists of all families $y=(a_q)_{q \in \QQ}$ such that
\[
\|y\|_g^p = \sup\left\{ \sum_{i=1}^N |a_{q_i}|^p w_i \colon q_1<\cdots q_i<\dots <q_N,\, N \in \NN\right\} <\infty.
\]
We gather together some properties of $g(\QQ,\ww,p)$. Throughout this section, $(\ssb_j)_{j=1}^\infty$ will denote the unit vector system of $g(\ww,p)$, while $(\ee_q)_{q\in\QQ}$ will denote the unit vector system of $g(\QQ,\ww,p)$. From now on, given $A\subseteq\QQ$ and a family $(a_q)_{q\in A}$, the formal series $\sum_{q\in A} a_q\, \ee_q$ denotes the vector $(c_q)_{q\in \QQ}$ in $\FF^\QQ$ defined by $c_q=a_q$ if $q\in A$ and $c_q=0$ otherwise. If $\phi\colon A\to \QQ$ is increasing we define
\[
U_\phi\colon \{ y \in \FF^\QQ \colon \supp(y)\subseteq A \} \to \FF^\QQ , \quad
\sum_{q\in A} a_q\, \ee_q\mapsto \sum_{q\in A} a_q \, \ee_{\phi(q)}.
\]
We say that $z\in\FF^\QQ$ is a shift of $y\in\FF^\QQ$ if $z=U_\phi(y)$ for some $\phi\colon\supp(y)\to\QQ$ increasing.  Similarly, we say that $z\in\FF^\QQ$ is a shift of $x\in\FF^\NN$ if  $z=I_\phi(x)$ for some $\phi\colon\NN\to\QQ$ increasing, where
\begin{equation*}
I_\phi\colon \FF^\NN \to \FF^\QQ , \quad
(a_j)_{j=1}^\infty \mapsto \sum_{j=1}^\infty a_j \, \ee_{\phi(j)}.
\end{equation*}

\begin{lemma}\label{lem:Ybasicproperties}
Let $\ww\in\WW$ and $0<p<\infty$. Then
\begin{enumerate}[label=(\alph*)]
\item\label{lem:Ybasicproperties:a} Given $y=(a_q)_{q \in \QQ}$,
\[
\|y\|_g^p = \sup\left\{ \sum_{i=1}^N |a_{q_i}|^p w_i \colon q_1<\cdots q_i<\dots <q_N, \, a_{q_i}\not=0, \, N \in \NN\right\}.
\]
\item\label{lem:Ybasicproperties:b} $\Vert y \Vert_g =\sup \{ \Vert y \chi_A\Vert_g \colon A\subseteq \QQ \text{ finite}\}$ for every $y\in \FF^\QQ$.

\item\label{lem:Ybasicproperties:c} $g(\QQ,\ww,p)$ is a $p$-convex quasi-Banach lattice with constant $1$.

\item\label{lem:Ybasicproperties:d} If $z\in\FF^\QQ$ is a shift of  $x\in  g(\ww,p)$, then $z\in g(\QQ,\ww,p)$, and $\Vert z\Vert_g=\Vert x\Vert_g$.

\item \label{lem:Ybasicproperties:e} $g(\QQ,\ww,p)$ has the following subsymmetry property: for every $\phi\colon A\subseteq \QQ \to \QQ$ increasing $U_\phi$ restricts to an isometric embedding from $\{ y\in g(\QQ,\ww,p) \colon \supp(y)\subseteq A\}$ into $g(\QQ,\ww,p)$.
\end{enumerate}
\end{lemma}
\begin{proof}
\ref{lem:Ybasicproperties:a} follows from the monotonicity of $\ww$.

\ref{lem:Ybasicproperties:b}, \ref{lem:Ybasicproperties:c} and \ref{lem:Ybasicproperties:d} are straightforward from the definition, and \ref{lem:Ybasicproperties:e} can be inferred from \ref{lem:Ybasicproperties:b} and \ref{lem:Ybasicproperties:d}.
\end{proof}

As we will prove in the forthcoming Section~\ref{prop:notsym}, the unit vector system, which is a $1$-unconditional basic sequence of $g(\QQ,\ww,p)$, does not span the whole space $g(\QQ,\ww,p)$. So, we define $g_0(\QQ,\ww,p)$ as the closure of $c_{00}(\QQ)$ in $g(\QQ,\ww,p)$. Note that a formal series $y=\sum_{q\in A}^\infty a_q \, \ee_{q}$ converges in $g(\QQ,\ww,p)$ if and only if $y\in g_0(\QQ,\ww,p)$.

\begin{lemma}\label{lem:additionalbasicproperty}Let $\ww\in\WW$ and $0<p<\infty$. Then for all $y=(a_q)_{q \in \QQ} \in g_0(\QQ,\ww,p)$ and $\varepsilon>0$, there exists $M \in \NN$ such that
\[
\sum_{i=1}^N |a_{q_i}|^p w_{M+i-1} < \varepsilon
\]
for all $N\in\NN$ and $q_1<\cdots <q_i<\cdots <q_N$.
\end{lemma}
\begin{proof}There is $A \subset \QQ$ finite such that $\|y - \sum_{q \in A} a_{q} \, \ee_q\|< 2^{-1/p} \varepsilon$. Now choose $M\ge1$ such that $w_M(\sum_{q \in A}|a_q|^p )< \varepsilon^p/2$. Then for all $N \in\NN$ and $(q_i)_{i=1}^N$ increasing we have
\begin{align*}
\sum_{i=1}^N |a_{q_i}|^pw_{M-1+i}
&=\sum_{i=1}^N |a_{q_i}|^p \chi_A(q_i) w_{M-1+i} + \sum_{i=1}^N |a_{q_i}|^p \chi_{A^c}(q_i) w_{M-1+i} \\
&\le w_M \sum_{i=1}^N |a_{q_i}|^p \chi_A(q_i) + \sum_{i=1}^N |a_{q_i}|^p \chi_{A^c}(q_i) w_{i} \\
&\le w_M\left(\sum_{q \in A}|a_q|^p \right) + \left\|y - \sum_{q \in A} a_{q} e_q\right\|_g^p \\
&< \frac{\varepsilon^p}{2} + \frac{\varepsilon^p}{2} = \varepsilon^p.\qedhere
\end{align*}
\end{proof}

\subsection{A dichotomy theorem for Garling sequence spaces.}\

\noindent In this section we will establish a correspondence between positionings and one-to-one sequences of rational numbers. This will be done in Lemma~\ref{lem:PvsS} below.

Given a positioning $\eta=(d_n)_{n=1}^\infty$ with $\pi[\eta]=(\pi_n)_{n=1}^\infty$ we recursively construct the sequence $q[\eta]=(q_n)_{n=1}^\infty$ in $\QQ\cap(0,1)$ by
\[
q_n=\frac{q_a+q_b}{2},
\]
where $a$ and $b$ in $\NN[n-1]\cup\{0,\infty\}$ are such that $\pi_{n-1}(a)=d_n-1$ and $\pi_{n-1}(b)=d_n$ ($b=\infty$ if $d_n=n$), with the convention $q_0=0$ and $q_\infty=1$.
We say that a sequence $(r_n)_{n=1}^\infty$ in $\QQ$ is \emph{compatible} with the positioning $\eta$ if
\[
r_{\pi_n^{-1}(1)} < \cdots <r_{\pi_n^{-1}(i)}<\cdots< r_{\pi_n^{-1}(n)}
\]
for every $n\in\NN$; in other words, $r_i<r_j$ whenever $n\in\NN$ and $i$, $j\in\NN[n]$ are such that $\pi_n(i)<\pi_n(j)$.

\begin{lemma}\label{lem:PvsS}Any one-to-one sequence in $\QQ$ is compatible with a unique positioning. Conversely, given a positioning $\eta$, the sequence $q[\eta]=(q_n)_{n=1}^\infty$ is compatible with $\eta$, and it is essentially unique with this property in the following sense: the one-to-one sequence $(r_n)_{n=1}^\infty$ also is compatible with $\eta$ if and only if the map $\phi\colon\{ q_n \colon n\in\NN\}\to\QQ$ given by $r_n=\phi(q_n)$ for all $n\in\NN$ is increasing.
\end{lemma}

\begin{proof}Given a  a one-to-one sequence $\rr=(r_n)_{n=1}^\infty$ in $\QQ$ we define a positioning $\eta[\rr]=(d_n[\rr])_{n=1}^\infty$ by
\[
d_n[\rr]=|\{ j\in\NN[n] \colon r_j\le r_n \}|, \quad n\in\NN.
\]
By equation \eqref{eq:etafrompi}, $\rr$ is compatible with a positioning $\eta$ if and only if $\eta=\eta[\rr]$.

A straightforward induction arguments yields that $q[\eta]$ is compatible with the positioning $\eta$. Finally, the very definition of compatibility yields that $\rr$ is compatible with $\eta$ if and only if the map $\phi$ is increasing on $\{q_j \colon 1\le j \le n \}$ for all $n\in\NN$.
\end{proof}


Now, we say that a vector $y\in \FF^\QQ$ is compatible with a seed $(f, \eta)$ if there is a sequence $(q_n)_{n=1}^\infty$ compatible with $\eta$ such that
\[
y=\sum_{n=1}^\infty a_n \, \ee_{q_n}.
\]
If two vectors $y_1$ and $y_2$ are compatible with the same seed, then $y_2$ is a shift of $y_1$.

We say that $y\in\FF^\QQ$ is compatible with $f\in\FF^\NN$ if it is compatible with the seed $(f,\eta_0)$, where $\eta_0$ is the trivial positioning.



\begin{lemma}\label{lem:garlingq1} Let $0<p<\infty$ and $\ww\in\WW$. Suppose that $(f,\eta)$ is a seed for $g(\ww,p)$. Let $f=(a_n)_{n=1}^\infty$. If $(q_n)_{n=1}^\infty$ is compatible with $\eta$ then
\[
E_m[f,\eta]=\left\Vert \sum_{j=m+1}^\infty a_j \, \ee_{q_j}\right\Vert_g
\]
for every $m\in\NN\cup\{0\}$.
\end{lemma}
\begin{proof}Let $\uu[f,\eta]=(u_n)_{n=1}^\infty$ and $\uu_l[f,\eta]=(u_{m,n})_{m\le n}$. Set also $u_0=0$.
By Lemma~\ref{lem:Ybasicproperties}~\ref{lem:Ybasicproperties:d},
\[
\Vert u_n-u_{m,n}\Vert_g=\left\Vert \sum_{j=m+1}^n a_j \, \ee_{q_j}\right\Vert_g, \quad 0\le m\le n.
\]
We finish the proof by applying Lemma~\ref{lem:Ybasicproperties}~\ref{lem:Ybasicproperties:b}.
\end{proof}

\begin{proposition}\label{prop:Yseeds}Let $0<p<\infty$ and $\ww\in\WW$. Let $y\in\FF^\QQ$ be compatible with a seed $(f,\eta)$. Then:
\begin{enumerate}[label=(\alph*)]
\item\label{prop:Yseeds:a} $E_0[f,\eta]<\infty$ if and only if $y\in g(\QQ,\ww,p)$, and
\item\label{prop:Yseeds:b} $E_\infty[f,\eta]=0$ if and only if $y\in g_0(\QQ,\ww,p)$.
\end{enumerate}
\end{proposition}
\begin{proof}It is immediate from Lemma~\ref{lem:garlingq1}.
\end{proof}

A sequence $(\yy_k)_{k=1}^\infty$ in $\FF^\QQ$ is said to be \emph{generated on $\QQ$ by a seed $(f,\eta)$} if for every $k\in\NN$, $\supp(\yy_k) \prec \supp(\yy_{k+1})$ and $\yy_k$ is compatible with $(f,\eta)$. If $\eta=\eta_0$ is the trivial positioning, we say that $(\yy_k)_{k=1}^\infty$ is generated on $\QQ$ by $f$. By Proposition~\ref{prop:Yseeds}, the sequence $(\yy_k)_{k=1}^\infty$ generated by $(f,\eta)$ belongs to $g(\QQ,\ww,p)$ (resp.\ $g_0(\QQ,\ww,p)$) if and only if $E_0[f,\eta]<\infty$ (resp.\ $E_\infty[f,\eta]=0$). By Proposition~\ref{prop:2}, a sequence $\BB$ generated on $\QQ$ by a vector $f$ belongs to $g(\QQ,\ww,p)$ if and only if $f\in g(\ww,p)$, in which case $\BB$ is contained in $g_0(\QQ,\ww,p)$. Let us construct a precise sequence generated on $\QQ$ by the seed $(f,\eta)$. If $q[\eta]=(q_n)_{n=1}^\infty$ and $f=(a_n)_{n=1}^\infty$ we define
\[
\BB_\QQ[f,\eta]=\left(\sum_{n=1}^\infty a_n \, \ee_{k-1+q_n}\right)_{k=1}^\infty.
\]
If $\eta_0$ is the trivial positioning we denote $\BB_\QQ[f]=\BB_\QQ[f,\eta_0]$.

In the case when $E_0[f,\eta]<\infty$, a sequence generated on $\QQ$ by the seed $(f,\eta)$ is a disjointly supported basic sequence of $g(\QQ,\ww,p)$.
If $\BB=(\xx_k)_{k=1}^\infty$ and and $\BB'=(\yy_k)_{k=1}^\infty$ are generated on $\QQ$ by the seed $(f,\eta)$ there is an increasing map $\phi\colon\cup_{k=1}^\infty\supp(\xx_k) \to \QQ$ such that $U_\phi(\xx_k)=\yy_k$ for every $k\in\NN$. Hence, by Lemma~\ref{lem:Ybasicproperties}~\ref{lem:Ybasicproperties:e}, $\BB$ and $\BB'$ are isometrically equivalent basic sequences when regarded in $g(\QQ,\ww,p)$. Moreover, also by Lemma~\ref{lem:Ybasicproperties}~\ref{lem:Ybasicproperties:e}, they are $1$-subsymmetric.

\begin{proposition}\label{lem:Y5} Let $0<p<\infty$ and $\ww\in\WW$. Suppose that $\BB$ is a basic sequence generated on $\QQ$ by a seed $(f,\eta)$ with
$E_\infty[f,\eta]=0$. Then $\BB$
belongs to the equivalence class of subsymmetric bases of $g(\ww,p)$ generated by $(f,\eta)$.
\end{proposition}
\begin{proof}
By Lemma~\ref{SubSymFromSeed}~\ref{SubSymFromSeed:b} and Proposition~\ref{SubSymFromamenableSeed}, it suffices to prove that if $\nu=(n_k)_{k=1}^\infty$ is such that
\[
\sum_{k=1}^\infty (E_{n_k}[f,\eta])^{\overline p}<\infty,
\]
then $ \BB[f,\eta,\nu]\simeq \BB_\QQ[f,\eta]$.

Put $f=(a_j)_{j=1}^\infty$, $q[\eta]=(q_j)_{j=1}^\infty$ and $\BB_\QQ[f,\eta]=(\yy_k)_{k=1}^\infty$. Let $\BB''=(\zz_k)_{k=1}^\infty$ be the basic sequence in $g(\QQ,\ww,p)$ defined by
\[
\zz_k=\sum_{j=1}^{n_k} a_j \, \ee_{k-1+q_j}, \quad k\in\NN,
\]
By Lemma~\ref{lem:Ybasicproperties}~\ref{lem:Ybasicproperties:d}, $\BB[f,\eta,\nu]\simeq\BB''$ isometrically. Using Lemma~\ref{lem:Ybasicproperties}~\ref{lem:Ybasicproperties:e} and Proposition~\ref{prop:Yseeds}~\ref{prop:Yseeds:b}, we obtain
\[
\Vert \yy_k-\zz_k\Vert_g=\left\Vert \sum_{j=1+n_k}^\infty a_j \, \ee_{k-1+q_j}\right\Vert=E_{n_k}[f,\eta].
\]
Then, by the principle of small perturbations, $\BB''\simeq \BB_\QQ[f,\eta]$.
\end{proof}

We are now ready to tackle our dichotomy result.

\begin{theorem}\label{thm:dochotomyGarling} Let $0< p<\infty$ and $\ww\in\WW$. A basic sequence $\BB$ is equivalent to a subsymmetric basic sequence of $g(\ww,p)$ if and only if
\begin{enumerate}[label=(\alph*)]
\item\label{thm:dochotomyGarling:a} Either $\BB$ is equivalent to a basic sequence generated on $\QQ$ by a seed with $E_\infty[f,\eta]=0$, or
\item\label{thm:dochotomyGarling:b} $\BB$ is equivalent to the unit vector system of $\ell_p$.
\end{enumerate}
\end{theorem}

\begin{proof}
Our arguments rely on the dichotomy provided by Theorem~\ref{theorem:trichotomy:subsym}. If $\BB$ is equivalent to a basic sequence generated by a seed $(f,\eta)$ with $E_\infty[f,\eta]=0$, the result follows from Proposition~\ref{lem:Y5}. Assume that there is a uniformly null block basic sequence $\BB'$ of $g(\ww,p)$ such that $\BB'\lesssim \BB$. Then, on one hand, by Theorem~\ref{thm:garling1}~\ref{thm:garling1:d}, passing to a suitable subsequence, $ \EE[\ell_p]\lesssim \BB'$. On the other hand, by Theorem~\ref{thm:garling1}~\ref{thm:garling1:b}, $\BB\lesssim \EE[\ell_p]$. Consequently, $\BB\simeq \EE[\ell_p]$. Finally, we note that Theorem~\ref{thm:garling1} also yields the existence of a basic sequence of $g(\ww,p)$ equivalent to $\EE[\ell_p]$.
\end{proof}

Let us briefly discuss the subsymmetric basic sequence structure of sequence Lorentz spaces. The symmetric basic sequences of $d(\ww,p)$, $p\ge 1$ were successfully studied in \cite{ACL1973}. A careful look at this paper reveals that the techniques developed there allow also to identify the subsymmetric basic sequences of $d(\ww,p)$. We remark that our techniques also apply to these spaces. We omit the straightforward details.
\begin{theorem}Let $\ww\in\WW$ and $0<p<\infty$. Every subsymmetric basic sequence of $d(\ww,p)$ is either equivalent to a basic sequence generated by a vector or equivalent to the unit vector basis of $\ell_p$.
\end{theorem}

\begin{theorem}\label{thm:symvssubsymLorentz} Let $0<p<\infty$ and $\ww\in\WW$. Every subsymmetric basic sequence of $d(\ww,p)$ is symmetric.
\end{theorem}

\begin{remark} To contextualize Theorem~\ref{thm:symvssubsymLorentz}, let us recall there are Banach spaces with a symmetric basis containing subsymmetric basic sequences which are not symmetric. In fact, Pe{\l}czy{\'n}ski's space with a universal unconditional basis (see, e.g., \cite{AK2016}*{Sect.\ 15.3}) has a symmetric basis and infinitely many subsymmetric basic sequences that are not symmetric.
\end{remark}

\subsection{Uniqueness of symmetric basic sequence}\label{GarlingSymSection}
In this section we deal with the symmetric basic sequence structure of Garling spaces. Let us bring forward the result which will allow us to tell apart subsymmetric basic sequences that are symmetric from those that are not.
\begin{proposition}\label{prop:notsym}Let $0<p<\infty$ and $\ww\in\WW$. Let $\BB$ be a basic sequence generated on $\QQ$ by a seed with $E_\infty[f,\eta]=0$. Then $\BB$ is not symmetric.
\end{proposition}

We emphasize that Proposition~\ref{prop:notsym} generalizes the main result of \cite{AALW2018}, where it is proved that the unit vector system of $g(\ww,p)$ is not a symmetric basis. However, the techniques we will use here are closer to those from \cite{AADK2019}.

Let $\tau\colon\QQ\to\QQ$ be the translation map given by $\tau(q)=1+q$. By Lemma~\ref{lem:Ybasicproperties}~\ref{lem:Ybasicproperties:e}, $U_\tau$ is an isometry on $g(\QQ,\ww,p)$.

\begin{lemma}\label{lem:steve1}Let $0<p<\infty$ and $\ww\in\WW$. Let $x \in g_0(\QQ,\ww,p)$ supported on $[0,\infty)$. Suppose that $(z_n)_{n=1}^\infty$ in $\FF^\QQ$ satisfies $\Vert z_n\Vert_g =1$, $\lim_n \Vert z_n\Vert_\infty=0$ and $\supp(z_n)\subseteq (-\infty,n)$. Then
\[
\lim_{n \to \infty} \left\| z_n +U_\tau^n(x) \right\|_g = \max\{ 1, \|x\|_g\}.
\]
\end{lemma}

\begin{proof}Set $x=(a_q)_{q\in\QQ}$ and let $\varepsilon>0$. By Lemma~\ref{lem:additionalbasicproperty} we can choose $L \ge 1$ such that
\[
\alpha := \sup\left\{ \sum_{i=1}^N |a_{q_i}|^pw_{L+i} \colon q_1<\cdots <q_i<\cdots <q_N, \, N\in\NN \right\} \le \varepsilon.
\]
Now choose $M \ge 1$ such that
\[
\Vert z_n\Vert_\infty \left(\sum_{i=1}^L w_i\right)^{1/p} \le \varepsilon, \quad \text{whenever}\, n\ge M.
\]

Let $n \ge M$. In order to estimate $A_n:=\left\| z_n +U_\tau^n(x) \right\|_g$ we pick an increasing sequence $(q_i)_{i=1}^N$ in $\QQ$. Let $j$ be the largest integer $i$ such that $q_i<n$.
Considering separately the cases when $j\le L$ and $j>L$, we obtain
\begin{align*}
A_n
&\le \max\left\{ \|U_\tau^n(x)\|_g + \Vert z_n\Vert_\infty \left(\sum_{i=1}^L w_i\right)^{1/p},
\alpha + \Vert z_n \Vert_g \right\} \\
&\le\max\{ \|x\|_g + \varepsilon, \alpha+1\}\\
&\le \varepsilon+\max\{ \|x\|_g, 1\}.
\end{align*}
Since the unit vector system of $g(\QQ,\ww,p)$ is $1$-unconditional, we have $\max\{ \|x\|_g, 1\} \le A_n$, so we are done.
\end{proof}

\begin{lemma}\label{lem:steve2} Let $0<p<\infty$ and $\ww\in\WW$. For $m\in\NN$ let $x \in g(\QQ,\ww,p)$ be supported on $(-\infty,m]$. Suppose that $(z_n)_{n=1}^\infty$ in $\FF^\QQ$ satisfies  $\lim_n \Vert z_n\Vert_\infty=0$, and $\Vert z_n\Vert_g = 1$ and $\supp(z_n)\subseteq [0,\infty)$ for every $n\in\NN$. Then
\[
\lim_{n \to \infty} \left\| x + U_\tau^m(z_n) \right\|_g =\left(1 + \|x\|_g^p\right)^{1/p}.
\]
\end{lemma}

\begin{proof} Set $x=(a_q)_{q\in\QQ}$ and let $\varepsilon>0$. By Lemma~\ref{lem:Ybasicproperties}~\ref{lem:Ybasicproperties:a}, there is $L\in\NN$ and an increasing $L$-tuple $(q_i)_{i=1}^L$ in $\QQ\cap(-\infty,m]$ such that
\[
\sum_{i=1}^L |a_{q_i}|^p w_i > \|x\|_g^p- \frac{\varepsilon}{3}.
\]
Now choose $N\in\NN$ such that, for every $n\ge N$,
\[
\Vert z_n\Vert_\infty^p \sum_{i=1}^L w_i < \frac{\varepsilon}{3}.
\]
Let $n\ge N$ and set $z_n=(c_q)_{q\in\QQ}$. Use Lemma~\ref{lem:Ybasicproperties}~\ref{lem:Ybasicproperties:a} to pick an increasing sequence $(r_i)_{i=1}^M$ in $[0,\infty)\cap \QQ$ such that
\[
\sum_{i=1}^M |c_{r_i}|^p w_i \ge \|z_n\|_g^p- \frac{\varepsilon}{3}=1-\frac{\varepsilon}{3}.
\]
Since the sequence $(q_1,\dots,q_L, m+r_1, \dots , m+r_M)$ is increasing, \eqref{eq:weight} yields
\begin{align*}
\Vert x + U_\tau^m(z_n)\Vert_g^p&\ge \|x\|_g^p- \frac{\varepsilon}{3} + 1-\frac{\varepsilon}{3}-\sum_{i=1}^N |c_{r_i}|^p (w_i-w_{i+L})\\
&\ge \|x\|_g^p- \frac{\varepsilon}{3} + 1-\frac{\varepsilon}{3} -\Vert z_n\Vert_\infty^p \sum_{i=1}^L w_i\\
&\ge \|x\|_g^p + 1-\varepsilon.
\end{align*}
Since $\Vert x + U_\tau^m(z_n)\Vert_g^p \le \|x\|_g^p + 1$ by Lemma~\ref{lem:Ybasicproperties}~\ref{lem:Ybasicproperties:c}, we are done.
\end{proof}

Given a seed $(f,\eta)$ with $E_\infty[f,\eta]=0$ and a sequence $\BB$ generated on $\QQ$ by $(f,\eta)$, we denote by $\Phi[f,\eta]$ the fundamental function of the $1$-subsymmetric basic sequence $\BB$ of $g_0(\QQ,\ww,p)$. If $\BB$ is generated by a vector $f\in g(\ww,p)$ we denote by $\Phi[f]$ the fundamental function of $\BB$. Notice that, if $C=\Vert f \Vert_\infty^{-1}$, by $1$-unconditionality and Lemma~\ref{lem:Ybasicproperties}~\ref{lem:Ybasicproperties:d},
$
\EE[g(\ww,p)] \lesssim_C \BB.
$
Hence,
\[
\left(\sum_{j=1}^nw_j\right)^{1/p} \le C \Phi[f,\eta](n), \quad n\in\NN,
\]
so that $\lim_n \Phi[f,\eta](n)=\infty$.

\begin{proposition}\label{prop:nonsymmetric}
Let $0<p<\infty$ and $\ww\in\WW$. Suppose $(\yy_k)_{k=1}^\infty$ is a basic sequence generated by a seed $(f,\eta)$ with $E_\infty[f,\eta]=0$. Put
\[
z_n=\frac{1}{\Phi[f,\eta](n)}\sum_{k=1}^n \yy_k, \quad n\in\NN.
\]
Then, for every $\varepsilon>0$ there exists $(n_j)_{j=1}^\infty$ such that for each $k\in\NN$ we have:
\begin{enumerate}[label=(\alph*)]
\item if $(z'_j)_{j=1}^k$ is a shift of $(z_{n_j})_{j=1}^k$ with $z'_1 \prec \cdots \prec z'_j\prec\cdots \prec z_k'$ then $\Vert \sum_{j=1}^k z_j' \Vert_g^p > k- \varepsilon$,

\item if $(z''_j)_{j=1}^k$ is a shift of $(z_{n_j})_{j=1}^k$ with $z''_k \prec \cdots \prec z''_j\prec\cdots \prec z_1''$ then $\Vert \sum_{j=1}^k z_j'' \Vert_g < 1+ \varepsilon$.
\end{enumerate}
Moreover, each $n_k$ can be chosen to be larger than a given function of $(n_1, \dots, n_{k-1})$.
\end{proposition}
\begin{proof}
We construct $(n_k)_{k=1}^\infty$ recursively. Choose $n_1\in\NN$ arbitrarily. Assume that $(n_j)_{j=1}^{k}$ is constructed. Since
\[
\Vert z_n\Vert_\infty=\frac{\Vert f\Vert_\infty}{\Phi[f,\eta](n)}\to 0
\]
and $\Vert z_n\Vert=1$ for every $n\in\NN$, we can apply Lemma~\ref{lem:steve2} with $m=m_k:=\sum_{j=1}^k n_j$ and
\[
x'=z_{n_1}+ U_\tau^{n_1}(z_{n_2})+\cdots+U_\tau^{n_1+\cdots + n_{k-1}}(z_{n_k}).
\]
We can also apply Lemma~\ref{lem:steve1} with
\[
x''= z_{n_k}+ U_\tau^{n_k}(z_{k-1})+\cdots+ U_\tau^{n_k+\cdots +n_2}(z_{n_1}).
\]
We infer that for $n\in\NN$ large enough $\Vert x'+U_\tau^{m}(z_n)\Vert_g^p>k+1-\varepsilon$ and
$\Vert z_n +U_\tau^n(x'')\Vert_g<1+\varepsilon$. Consequently, given $k\in\NN$, the $k$-tuples $(U_\tau^{m_{j-1}} (z_{n_j} ))_{j=1}^k$ and $( U_\tau^{m_k-m_j} (z_{n_j} ) )_{j=1}^k$ satisfy the desired conditions.
\end{proof}

We are ready to show Proposition~\ref{prop:notsym}. We will take care of this along with the proof of our next Proposition. Note the connection between this result and Theorem~\ref{thm:garling1}~\ref{thm:garling1:c}.

\begin{proposition}\label{cor:Ysubsymmetric} Let $0<p<\infty$ and $\ww\in\WW$. Suppose $\BB$ is a block basic sequence generated on $\QQ$ by a seed $(f,\eta)$ with $E_\infty[f,\eta]=0$. Then $\ell_\infty$ is block finitely representable on $\BB$.
\end{proposition}
\begin{proof}[Proof of Propositions~\ref{prop:notsym} and \ref{cor:Ysubsymmetric}]
Given $0<\varepsilon <1$ and $k\in\NN$, let
$
\BB'_{\varepsilon,k}:=(\xx_j')_{j=1}^k
$
and
$
\BB''_{\varepsilon,k}:=(\xx_j'')_{j=1}^k
$
be the $k$-tuples provided by Proposition~\ref{prop:nonsymmetric}. By $1$-unconditionality, $\BB''_{\varepsilon,k}$ is $(1+\varepsilon)$-equivalent to the unit vector basis of $\ell_\infty^k$. If $\BB$ were symmetric, $\BB'_{\varepsilon,k}$ would be uniformly equivalent to $\BB''_{\varepsilon,k}$. Hence there would be a uniform constant $C$ such that
\[
\left\Vert \sum_{j=1}^k x_j'\right\Vert_g \le C\left\Vert \sum_{j=1}^k x_j'' \right\Vert_g.
\]
Letting $k$ tend to infinity, we would reach an absurdity.
\end{proof}

Next we put to use our techniques for obtaining structural properties of Garling sequence spaces.

\begin{theorem}Let $0<p<\infty$ and $\ww\in\WW$. Suppose that $\BB$ is a symmetric basic sequence of $g(\ww,p)$. Then $\BB$ is equivalent to the unit vector basis of $\ell_p$.
\end{theorem}

\begin{proof}Just combine Theorem~\ref{thm:dochotomyGarling} with Proposition~\ref{prop:notsym}.
\end{proof}

\begin{theorem}Let $0<p<\infty$ and $\ww\in\WW$. Suppose that $\BB$ is a subsymmetric basic sequence of $g(\ww,p)$ nonequivalent to the unit vector system of $\ell_p$. Then $\ell_\infty$ is block finitely representable on $\BB$.
\end{theorem}

\begin{proof}Combine Theorem~\ref{thm:dochotomyGarling} with Proposition~\ref{cor:Ysubsymmetric}.
\end{proof}

\begin{corollary}Let $1\le p<\infty$ and $\ww\in\WW$. Suppose that $\BB$ is a subsymmetric basis equivalent to a basic sequence of both $g(\ww,p)$ and $d(\ww,p)$. Then $\BB$ is equivalent to the canonical basis of $\ell_p$. In particular, $d(\ww,p)$ is not isomorphic to a subspace of $g(\ww,p)$ and $g(\ww,p)$ is not isomorphic to a subspace of $d(\ww,p)$.
\end{corollary}

\begin{proof}Combine Theorem~\ref{thm:symvssubsymLorentz} and Proposition~\ref{prop:notsym}.
\end{proof}

We close this section with the aforementioned application to Garling sequence spaces over $\QQ$.

\begin{proposition} Let $1\le p<\infty$ and $\ww\in\WW$. Given $\varepsilon>0$ there is a disjointly supported basic sequence in $g_0(\QQ,\ww,p)$ which is $(1+\varepsilon)$-equivalent to the unit vector system of $c_0$. Consequently $\ell_\infty$ is $(1+\varepsilon)$-isomorphic to a subspace of $g(\QQ,\ww,p)$ and $g_0(\QQ,\ww,p)\subsetneq g(\QQ,\ww,p)$.
\end{proposition}

\begin{proof}Pick an arbitrary $0\not=y\in g_0(\QQ,\ww,p)$ (for instance, $y=\ee_{1/2}$) and let $(f,\eta)$ be  the seed with which $y$ is compatible,
so that $E_\infty[f,\eta]=0$ by Proposition~\ref{prop:Yseeds}. Let $(n_j)_{j=1}^\infty$ be the sequence provided by Proposition~\ref{prop:nonsymmetric}. For each $j\in\NN$ there is a shift $z_j''$ of $z_{n_j}$ such that $\supp(z_j'')\subseteq(2^{-j}, 2^{-j+1})$. Let $x=(a_j)_{j=1}^\infty\in \FF^\NN$. By $1$-unconditionality and Lemma~\ref{lem:Ybasicproperties}~\ref{lem:Ybasicproperties:b} and \ref{lem:Ybasicproperties:e}, the formal series $T(x):=\sum_{j=1}^\infty a_j \, z_{j}''$ satisfies
\[
\Vert x\Vert_\infty
\le \Vert T(x) \Vert_g
\le \Vert x\Vert_\infty\sup_k \left\Vert \sum_{j=1}^k \, z_{j}''\right\Vert_g
\le (1+\varepsilon) \Vert x\Vert_\infty.\qedhere
\]
\end{proof}

\subsection{Non-equivalent subsymmetric basic sequences}\label{GarlingSubSymSection}In this section we show another geometric difference between Lorentz sequence spaces and Garling sequence spaces. Recall that, as we pointed out in Sect.~\ref{intro}, some Lorentz sequence spaces have exactly two non-equivalent subsymmetric basic sequences.

\begin{theorem}\label{thm:infinitessymbases}Let $0< p<\infty$ and $\ww\in\WW$. The Garling sequence space $g(\ww,p)$ contains a continuum of inequivalent subsymmetric basic sequences which is totally ordered in the domination ordering. The Garling sequence space $g(\ww,p)$ also contains a continuum of subsymmetric basic sequences which are pairwise incomparable in the domination ordering.
\end{theorem}

Our proof of Theorem~\ref{thm:infinitessymbases} will rely on Proposition~\ref{prop:nonsymmetric} (or, alternatively, on
\cite{AADK2019}*{Lemma 2.3}) and the following set theory lemmas.

\begin{lemma}\label{Lemma:TotallyOrdered}There is an uncountable totally ordered subset $\AAA$ of $\PP(\NN)$ such that if $A$, $B\in\AAA$ are such that $A\subseteq B$ and $A\not=B$ then $B\setminus A$ is infinite.
\end{lemma}

\begin{lemma}\label{Lemma:TotallyDisjoint}There is an uncountable subset $\AAA$ of $\PP(\NN)$ such that for every $A$, $B\in\AAA$ with $A\not=B$ both $B\setminus A$ and $A\setminus B$ is infinite.
\end{lemma}

\begin{proof}[Proof of Lemmas~\ref{Lemma:TotallyOrdered} and \ref{Lemma:TotallyDisjoint}] Let $\DD$ be the set of all dyadic rationals in $(0,1]$. Obviously, if $A_0\subseteq A_1$ and $A_1\setminus A_0$ is infinite then there is $A_0\subseteq B \subseteq A_1$ such that both $B\setminus A_0$ and $A_1\setminus B$ are infinite. Thanks to this fact we recursively construct a function
$A\colon\DD\to\PP(\NN)$ such that $A(s)\subseteq A(t)$ and $A(t)\setminus A(s)$ whenever $s<t$. Define $B\colon (0,1]\to \PP(\NN)$ by
\[
B(s)=\cup \{ A(q) \colon q\in \DD, \, q<s \}.
\]
Let $0<s<t\le 1$. By construction, $B(s)\subseteq B(t)$. Since $\DD$ is dense in $(0,1]$, $B_{s,t}:=B(t)\setminus B(s)$ is infinite. This way, the set $\{B(s) \colon s\in (0,1]\}$ does the job claimed in Lemma~\ref{Lemma:TotallyOrdered}.

If $0<s<s'<t<t'\le 1$ then $B_{s,t}\setminus B_{s',t'}=B_{s,s'}$ and $B_{s',t'}\setminus B_{s,t}=B_{t,t'}$ and so both sets are infinite. Hence $\{B_{s/2,s}\colon s\in (0,1]\}$ makes Lemma~\ref{Lemma:TotallyDisjoint} good.
\end{proof}

\begin{proof}[Proof of Theorem~\ref{thm:infinitessymbases}] With the convention $\{0,1\}^0=\emptyset$, let $B_\infty = \cup_{k=0}^\infty \{0,1\}^k$ be the infinite binary tree with its usual partial order $\preceq$. Given $k\in\NN\cup\{0\}$ and $\tau=(b_1,\dots,b_k)\in B_\infty$, let $|\tau| := k$. If $k\ge 1$ and $b_k=0$ (resp.\ $b_k=1$) we shall say that $\tau$ \emph{ends with} $0$ (resp.\ $1$). A \emph{branch} $b$ of $B_\infty$ is a maximal totally ordered subset of $B_\infty$. We identify $b$ with the unique infinite sequence $(b_k)_{k=1}^\infty \in \{0,1\}^\NN$ such that $b = \{(b_1,\dots,b_k) \colon k \ge 0\}$.

We shall define a sequence $(\varepsilon_k)_{k=1}^\infty$ of positive numbers, two increasing sequences $\nu=(n_k)_{k=1}^\infty$ and $(m_k)_{k=1}^\infty$ of natural numbers, and a family of vectors $(u_\tau)_{\tau\in B_\infty}$ such that, with the conventions that $\varepsilon_0 =m_0=1$, $n_0=0$, and $u_\emptyset =\ee_1$, for any $\tau\in B_\infty$ with $|\tau|=k\ge 1$ we have
\begin{itemize}
\item If $\tau$ ends with $0$ then $u_\tau=0$,

\item $\|u_\tau\|_g \le 2^{-k} \varepsilon_{k-1}$,

\item $\varepsilon_k \le \min\{\varepsilon_{k-1}, 2(2^p-1)^{1/p}m_{k-1}^{-1/p}\}$,

\item $\supp(u_\tau)\subseteq \NN[n_{k}]\setminus \NN[n_{k-1}]$,

\item If $\tau$ ends with $1$ there is $(y_i)_{i=1}^{m_{k-1}}$ disjointly supported such that each $y_i$ is a coordinate projection of $u_\tau$ and
\[
\left\| \sum_{i=1}^{m_{k-1}} z_i \right\|_g \ge 2^{-k}\varepsilon_{k-1} m_{k-1}^{1/p}
\]
if $z_i\in\FF^\QQ$  is a shift of
 $y_i$ and $z_1\prec \cdots \prec z_i\prec \cdots \prec z_{m_{k-1}}$.

\item If $f_\tau=\sum_{\sigma\preceq \tau} u_\sigma$ then
\[
\Phi[f_\tau](m_{k}) \le \frac{2^{-k-1}\varepsilon_{k} m_{k}^{1/p}}{k}.
\]
\end{itemize}
To that end we proceed recursively. Suppose $k \ge 0$, and that $\varepsilon_i$, $n_i$, $m_i$ and $u_\tau$ have been defined for $0 \le i \le k$ and $|\tau| \le k$.

Suppose that $|\tau| = k+1$. If $\tau$ ends with $0$, we simply set $u_\tau = 0$. If $\tau$ ends with $1$ then we use Proposition~\ref{prop:nonsymmetric} with an arbitrary finitely supported $f\in g(\ww,p)$ (for instance $f=\ee_{1}$) and Lemma~\ref{lem:Ybasicproperties}~\ref{lem:Ybasicproperties:d} to select vectors $(y_i)_{i=1}^{m_k}$ in $g(\ww,p)$ satisfying
\begin{itemize}
\item $\supp y_i\cap \NN[n_{k}]=\emptyset$,

\item $ \| \sum_{i=1}^{m_{k}} y_i \| = 1$, and

\item $2 \| \sum_{i=1}^{m_{k}} z_i \| \ge m_{k}^{1/p}$ whenever $z_i$ is a shift of
$y_i$ and $z_1\prec \cdots \prec z_i\prec \cdots \prec z_{m_{k}}$.
\end{itemize}
Now define
\begin{equation*}
u_\tau :=2^{-k}\varepsilon_k \sum_{i=1}^{m_{k}} y_i.
\end{equation*}

Once $u_\tau$ is constructed for all $\tau$ with $|\tau|=k+1$, we choose $n_{k+1}$ to be the largest integer belonging to the support of some $u_\tau$.

Let $\varepsilon_{k+1} := \min\left\{\varepsilon_k, 2(2^p-1)^{1/p}m_{k}^{-1/p}\right\}$. Note that, by Propositions~\ref{Prop:EventuallyNull} and \ref{lem:Y5}, $\BB_\QQ[f_\tau]\simeq \EE[g(\ww,p)]$ whenever $|\tau|=k$. Hence,
\[
\lim_{m\to\infty} \sup_{|\tau|=k}\frac{\Phi[f_\tau](m)}{m^{1/p}} = 0.
\]
Thus, if we pick $\delta= 2^{-k-2} (k+1)^{-1} \varepsilon_{k+1}$, there is $m=m_{k+1}>m_{k}$ such that $ m^{-1/p} \Phi[f_\tau](m) \le\delta $ whenever $|\tau|=k$. This ends the recursive definition.

Given a branch $b$ of $B_\infty$, $(u_\tau)_{\tau\in b}$ is a disjointly supported sequence in $g(\ww,p)$. Since
\[
\sum_{\tau\in b} \Vert u_\tau\Vert_g^p \le\sum_{k=0}^\infty 2^{-kp}<\infty,
\]
the series $\sum_{\tau\in b} u_\tau$ converges and so we can define $h_b\in g(\ww,p)$ by
\[
h_b:=\sum_{\tau\in b} u_\tau.
\]
By Proposition~\ref{lem:Y5}, the basic sequence $\BB_\QQ[h_b]$ generated on $\QQ$ by $h_b$ belongs to the equivalence class of subsymmetric basic sequences of $g(\ww,p)$. We shall tell apart different elements in the family
\[
(\BB_\QQ[h_b] \colon b \text{ branch of } B_\infty)
\]
by comparing the fundamental functions $(\Phi[h_b] \colon b \text{ branch of } B_\infty)$.

Let $b= (b_j)_{j=1}^\infty$ be a branch and denote $\BB_\QQ[h_b]=(\yy_j)_{j=1}^\infty$. Fix $k\ge 0$.

On one hand suppose that $b_{k+1} = 1$. Then there is $z_1\prec \cdots \prec z_i\prec \cdots \prec z_{m_{k}}$ such that $z_i$ is a coordinate projection of $\yy_i$ for every $i=1$, \dots, $m_k,$ and
\[
\left \| \sum_{i=1}^{m_{k}} z_i \right\| \ge A_k:=2^{-k-1} \varepsilon_{k} m_k^{1/p}.
\]
Then $ \sum_{i=1}^{m_{k}} z_i$ is a coordinate projection of $\sum_{i=1}^{m_{k}} \yy_i$. Therefore, by $1$-unconditionality,
\[
\Phi[h_b](m_k) \ge \left\| \sum_{i=1}^{m_k} z_i\right\| \ge A_k.
\]

Suppose, on the other hand, that $b_{k+1} = 0$. Then, if we denote $\tau=(b_j)_{j=1}^{k+1}$, we have
\begin{align*}
\Vert h_b -f_\tau\Vert_g^p&=\left\Vert\sum_{\substack{\tau\in b\\ |\tau|\ge k+2}} u_{\tau}\right\Vert_g^p \\
\le&\sum_{\substack{\tau\in b\\ |\tau|\ge k+2}} \Vert u_{\tau}\Vert_g^p \\
&\le \sum_{n=k+2}^\infty 2^{-np} \varepsilon_{n-1}^p\\
&\le \frac{2^{p}(2^p-1)}{m_k} \sum_{n=k+2}^\infty 2^{-np}\\
&= \frac{2^{-kp}}{m_k}.
\end{align*}
Consequently, for every $m\in\NN$,
\[
(\Phi[h_b](m))^p\le \frac{m}{m_k}2^{-kp} +(\Phi[f_\tau](m))^p.
\]
Choosing $m=m_k$ we obtain
\[
\Phi[h_b](m_k)\le B_k:=(k^{-p} A_k^p + 2^{-kp})^{1/p}.
\]

A branch $b$ is univocally determined by the set $A(b) = \{k \ge 0\colon b_{k+1} = 1\}$. Conversely, given $A\subseteq\NN\cup\{0\}$, there is a branch $b$ such that $A(b)=A$. If $b$ and $b'$ are branches with $A(b)\setminus A(b')$ infinite then, since $\lim_k A_k/B_k=\infty$,
\[
\sup_m \frac{\Phi[h_b](m)}{\Phi[h_{b'}](m)} \ge \sup_{k\in A(b)\setminus A(b')} \frac{\Phi[h_b](m_k)}{\Phi[h_{b'}](m_k)}
\ge \sup_{k\in A(b)\setminus A(b')} \frac{A_k}{B_k} =\infty.
\]
We infer that $\BB[h_{b'}]$ does not dominate $\BB[h_b]$. So, if both $A(b)\setminus A(b')$ and $A(b')\setminus A(b)$ are infinite, $\BB[h_b]$ are $\BB[h_{b'}]$ are incomparable. Now the latter part of the theorem follows from Lemma~\ref{Lemma:TotallyDisjoint}.

The lattice structure of $\BB(\QQ,\ww,p)$ yields $\BB[h_b]\lesssim_1 \BB[h_{b'}]$ if $A(b)\subseteq A(b')$. So, the former part of the theorem follows from Lemma~\ref{Lemma:TotallyOrdered}.
\end{proof}

\subsection*{Acknowledgment}
The authors would like to thank the anonymous referees for their insight and their helpful suggestions.

\end{document}